\documentclass[11pt, reqno]{amsart}
\usepackage[utf8]{inputenc}
\usepackage{amsthm, amsmath, amsfonts, amssymb, amsbsy}
\usepackage{tikz-cd}
\usepackage{cite}
\usepackage[all]{xy}
\usepackage{csquotes}
\usepackage{hyperref}
\usepackage{mathrsfs}
\usepackage{mathtools}
\usepackage{url}
\usepackage{thmtools,xcolor}

\usepackage{xifthen}
\newcommand*{\msc}[2][]{\href{https://mathscinet.ams.org/mathscinet/search/mscdoc.html?code=#2,(#1)}{Primary: #2\ifthenelse{\isempty{#1}}{}{; Secondary: #1}.}}

\declaretheoremstyle[
  headfont=\color{blue}\normalfont\bfseries,
  bodyfont=\color{blue}\normalfont\itshape,
]{colored}

\theoremstyle{plain}
\newtheorem{thm}{Theorem}
\newtheorem{lemma}[thm]{Lemma}
\newtheorem{prop}[thm]{Proposition}
\newtheorem{cor}[thm]{Corollary}

\theoremstyle{definition}
\newtheorem{defi}{Definition}
\theoremstyle{remark}
\newtheorem{exam}{Example}

\newtheorem{rema}{Remark}
\newtheorem{question}{Question}

\newtheoremstyle{cited}%
  {3pt}% (space above)
  {3pt}% (space below)
  {\itshape}% (body font)
  {}% (indent amount)
  {\bfseries}% {theorem head font}
  {.}% {punctuation after theorem head}
  {.5em}% {space after theorem head}
  {\thmname{#1} \thmnumber{#2} \thmnote{\normalfont#3}}% {theorem head spec}
\theoremstyle{cited}

\theoremstyle{cited}

\newtheorem{citedlemma}[thm]{Lemma}
\newtheorem{citedprop}[thm]{Proposition}

\newtheorem{theoA}{Theorem}

\newtheorem{corA}[theoA]{Corollary}

\newcommand{\C}{\ensuremath{\mathcal{C}}}
\newcommand{\D}{\ensuremath{\mathcal{D}}}
\newcommand{\E}{\ensuremath{\mathcal{E}}}
\renewcommand{\P}{\ensuremath{\mathcal{P}}}

\newcommand{\F}{\ensuremath{\mathcal{F}}}
\newcounter{para}

\DeclareMathOperator{\Ob}{Ob}
\DeclareMathOperator{\Res}{\downarrow}
\DeclareMathOperator{\Module}{-Mod}
\DeclareMathOperator{\gldim}{l.gl.dim}
\DeclareMathOperator{\Tor}{Tor}
\DeclareMathOperator{\Ext}{Ext}
\DeclareMathOperator{\Hom}{Hom}
\DeclareMathOperator*{\colim}{colim}
\DeclareMathOperator{\pr}{pr}
\DeclareMathOperator{\Mod}{-Mod}
\DeclareMathOperator{\SL}{SL}

\title[Hereditary infinite category algebras]{Sufficient and necessary conditions for hereditarity of infinite category algebras}

\author{Malte Lackmann}
\address{Mathematisches Institut, Universit\"{a}t Bonn, Germany.}
\email{lackmann@math.uni-bonn.de}

\author{Liping Li}
\address{LCSM (Ministry of Education), School of Mathematics and Statistics, Hunan Normal University, Changsha, Hunan 410081, China.}
\email{lipingli@hunnu.edu.cn}

\thanks{The first author is supported by the ERC Advanced Grant "KL2MG-interactions" (no. 662400) of Wolfgang L\"uck. The paper was written in his time as PhD student of Wolfgang L\"uck and may have overlap with his thesis. The second author is supported by the National Natural Science Foundation of China (Grant No. 11771135), the HuXiang High-Level Talents Gathering Project of Hunan Provincial Science and Technology Department (Grant No. 2019RS1039), and the Research Foundation of Hunan Provincial Education Department (Grant No. 18A016).}

\begin{document}

\maketitle

\begin{abstract}
We describe necessary and sufficient conditions for the hereditarity of the category algebra of an infinite EI category satisfying certain combinatorial assumptions. More generally, we discuss conditions such that the left global dimension of a category algebra equals the maximal left global dimension of the endomorphism algebras of its objects, and classify its projective modules in this case. As applications, we completely classify transporter categories, orbit categories, and Quillen categories with left hereditary category algebras over a field.
\end{abstract}

\section{Introduction}

\subsection{Motivation} EI categories, in which every endomorphism is an isomorphism, have widely appeared in representation theory, topology, and the interface between the two: Keywords are Bredon coefficient systems and homology theories \cite{bredon, davis-lueck, trafo, symonds}, approximation of classifying spaces of groups \cite{blo, jmo}, proofs of the Farrell-Jones conjecture \cite{davis-lueck, BLR}, cohomology theory \cite{xu-transporter}, fusion systems \cite{ao, linckelmann1}, reformulations of Alperin's weight conjecture \cite{linckelmann2, webb}, etc. Major examples of EI categories appearing in these applications include the transporter category of a poset with a group action, the orbit category of a group with respect to a family of subgroups, Quillen categories (also called Frobenius categories), and fusion systems. The representation theory of EI categories embraces the representation theory of groups, quivers, and posets, and is of interest to experts in group theory, algebraic topology and cohomology theory.

The original interest of this paper comes from the quest to describe explicit upper bounds of the global dimension of the left module categories of some specific EI categories appearing in topology; in particular, we sought to classify or characterise those EI categories whose left module categories are hereditary.
In \cite{li}, the second author gave a characterisation of \emph{finite} EI categories whose category algebra over a field is hereditary. However, since many examples appearing in group theory and algebraic topology (for instance, fusion systems of infinite groups) are infinite categories, one is led to consider the problem in a much more general framework, and accordingly, a much more difficult situation.

As for a particular application, the first author's results in \cite{SW} about the existence of Chern characters for rational $G$-homology theories can now be generalised to proper $G$-homology theories with $G$ an infinite discrete group, using Theorem~\ref{orbit-hereditary} below. Furthermore, using the content of this paper, it is possible to generalise certain finite results of Chen and Wang \cite{cw} about an algebraic enrichment of the well-known correspondence between symmetrisable Cartan matrices and graphs with automorphisms to the infinite case under some extra efforts.

\subsection{Main results and strategies}

Loosely speaking, compared to a group, the combinatorial structure of an EI category $\C$ has two more patterns:
\begin{itemize}
    \item factorisation properties of non-invertible morphisms between pairs of non-isomorphic objects,
    \item the biset structures of non-invertible morphisms between pairs $(c,d)$ of non-isomorphic objects  under the actions of the endomorphism groups $G_c$ and $G_d$.
\end{itemize}
To bound the global dimension of the left module category of $\C$, we need to impose certain extra conditions on these two patterns. For the first one, we introduce the finite factorisation property FFP (see Definition \ref{infinite-compositions}) and the unique factorisation property UFP (see Definition \ref{UFP}); for the second one, we introduce Conditions $(\mathrm{A}_N)$, $(\mathrm{B}_N)$, and $(\mathrm{BP})$, cf.\ Subsection \ref{ANBN}. These conditions are formulated for arbitrary directed $k$-linear categories $\C$. They are most general, but slightly artificial, and can be replaced by two much more transparent conditions (A) and (B), cf.\ \ref{k-linear categories}, in the case that $N=1$, which in turn can be translated into combinatorial conditions $(\mathrm{A}_d)$ and $(\mathrm{B}_d)$, cf.\ Subsection \ref{discrete ei categories}, if $\C$ is a discrete EI category. Note that for $\C$ finite, the question whether $k\C$ is hereditary is fully decided by the first pattern, at least if $k$ is a field \cite{li}, while the second pattern is not visible through the eyes of algebra since the group rings $kG_c$ have global dimension either $0$ or $\infty$.
Contrarily, the second pattern plays an important role in the infinite case.

The following theorem for EI categories is a weaker version of Theorem \ref{sufficiency-ANBN}, which bounds the left global dimension of $k$-linear categories satisfying certain conditions.

\begin{theoA} \label{bound gldim}
Let $N$ be a positive integer, $k$ a commutative ring and $\C$ a discrete EI category with the FFP and the UFP. Suppose that the global dimension of the group ring $kG_c$ is bounded by $N$ for every object $c$, and that $k\C$ satisfies Conditions $(\mathrm{A}_N)$, $(\mathrm{B}_N)$ and $(\mathrm{BP})$. Then the global dimension of the category of left $k\C$-modules is bounded by $N$ as well.
\end{theoA}

We prove Theorem \ref{bound gldim} using a result of Cuntz and Quillen \cite{cuntz-quillen} on modules of relative differential forms, adapted for rings with approximate unit. Without Conditions $(\mathrm{A}_N)$, $(\mathrm{B}_N)$ and $(\mathrm{BP})$ or  their twin conditions, this method would yield an upper bound of $N+1$ for the global dimension \cite[Thm.~2.2.11]{hazewinkel-gubareni-kirichenko}. Thus this paper is all about decreasing this bound by $1$, which makes of course a great difference in applications, in particular for small values of $N$.

We give an alternative, combinatorial proof of Theorem \ref{bound gldim} in Section \ref{suff2}. This proof needs an additional condition, but we get the following additional statement as an upshot:

\begin{theoA} \label{intro-sum-of-rep}
In the situation of Theorem \ref{bound gldim}, suppose additionally that the object poset of $\C$ has no left-infinite chains (see Def. \ref{left-infinite chains}). Then every projective left $k\C$-module is isomorphic to a direct sum of induced modules of the form $k\C\otimes_{kG_c} P_c$, where $c\in \mathrm{Ob}(\C)$ and $P_c$ is a projective left $kG_c$-module.
\end{theoA}

This is a result of Lück \cite[Cor.~9.40]{trafo} in  the case of finite type projectives (with $\C$ an arbitrary EI category), but a non-trivial new result in the above generality.

For $N = 1$, we can prove the converse implication of Theorem \ref{bound gldim}, significantly generalising \cite[Theorem 1.2]{li}.

\begin{theoA} \label{hereditary-minimal} Let $k$ be a semisimple commutative ring and $\C$ a discrete EI category with the finite factorisation property FFP.
The category of left $\C$-modules is hereditary if and only if the following conditions are satisfied:
\begin{enumerate}
\item the group ring $kG_c$ is hereditary for every object $c$ in $\C$,
\item $\C$ has the unique factorisation property UFP,
\item Conditions $(\mathrm{A}_d)$ and $(\mathrm{B}_d)$ hold.
\end{enumerate}
\end{theoA}

The proof of the necessity direction is based on the techniques introduced in \cite{li}, which we extend to the more general situation considered in this paper. Together with the classification of hereditary groups rings  by Dicks \cite{dicks}, which we cite in Proposition \ref{dicks}, and which can be seen as a precursor of our result, specifically the case that $\mathrm{Ob}(\C)$ is a singleton, the above theorem characterises a wide class of EI categories with hereditary left module category in terms of transparent and easy to check combinatorial conditions.

We shall also point out that there exist EI categories whose left module category is hereditary, but the right module category is not. This can easily be deduced by observing that the unique factorisation property UFP is symmetric (that is, $\C$ has the UFP if and only if its opposite category $\C^{\mathrm{op}}$ has the UFP as well), and Condition $(\mathrm{A}_d)$ is symmetric as well, but  $(\mathrm{B}_d)$ is not; see Example \ref{non-symmetric}. Their category algebras are thus  rings with approximate unit whose left and right global dimension differ, and which are consequently not approximately Noetherian, as discussed in Remark \ref{approximately Noetherian}.

By considering the opposite category, the above theorem immediately gives us a characterisation of EI categories with the FFP such that both module categories are hereditary.

\begin{corA} \label{two-sided hereditary} Let $k$ and $\C$ be as in Theorem \ref{hereditary-minimal}. Then both the categories of left and right $\C$-modules are hereditary if and only if the following conditions are satisfied:
\begin{enumerate}
\item the group ring $kG_c$ is hereditary for every object $c$ in $\C$,
\item $\C$ has the unique factorisation property UFP,
\item all biset stabilisers (cf. Subsection \ref{discrete ei categories}) are $k^\times$-finite.
\end{enumerate}
\end{corA}

\subsection{Applications}

As previously mentioned, we are interested in three concrete types of EI categories: transporter categories, orbit categories and Quillen categories (including fusion systems as a special case). Applying Theorem \ref{hereditary-minimal}, we work out for these categories when their left module categories over a field are hereditary.

A transporter category $\mathcal{P} \rtimes G$ is the Grothendick construction of a poset $\mathcal{P}$ equipped with the action of a group $G$. We translate the conditions FFP and UFP occurring in Theorem \ref{hereditary-minimal} to the two combinatorial conditions (ESC) and (USC) on the poset $\mathcal{P}$. It turns out that these, together with two conditions on the stabilisers, completely characterise transporter categories with hereditary left module categories over a field. For details, please refer to Subsection \ref{transporter categories} and Theorem \ref{transporter-hereditary}.

Given a discrete group $G$ and a family $\mathcal{F}$ of subgroups closed under conjugations and taking subgroups, one can define the orbit category and the Quillen category.

We find the correct group-theoretic conditions to be imposed on $G$ and the members of $\mathcal{F}$ such that the orbit category has a hereditary left module category over a field $k$ in Subsection \ref{orbit categories}:

\begin{theoA} \label{orbit-hereditary} Let $k$ be a field, $G$ a discrete group and $\F$ a family of finite subgroups. Then $k\mathrm{Or}(G,\F)$ is hereditary if and only if
\begin{itemize}
\item $G$ is either countable locally $k^\times$-finite or the fundamental group of a connected graph of $k^\times$-finite groups,
\item all members of $\F$ are cyclic of prime power order, invertible in $k$, and their Weyl groups are $k^\times$-finite (except possibly for the Weyl group of $\{1\}$).
\end{itemize}
\end{theoA}

This result generalises \cite[Thm.~D]{SW}. The discussion of Quillen categories, which is quite similar, is treated in Subsection \ref{quillen categories}, leading to Theorem \ref{quillen-hereditary}. Furthermore, for the above three concrete types of categories, we prove that their left module categories are hereditary if and only if so are the right module categories.

\subsection{Open questions}

We formulate some interesting open questions which follow naturally from our results. The first two refer to the goal to extend the 'if and only if' characterisation of Theorem \ref{hereditary-minimal} to greater generality.

\begin{question}
Are the conditions of Theorem \ref{bound gldim} necessary?
\end{question}

In detail, we mean the following: Suppose that the EI category $\C$ satisfies the FFP. Suppose that the global dimensions of the endomorphism algebras $R_x$ are bounded by some $N\ge 1$, and that $N$ is the least such upper bound. If $\C$ has global dimension $N$ as well, does $\C$ necessarily satisfy the UFP and the algebraic conditions $(\mathrm{A}_N)$, $(\mathrm{B}_N)$ and (BP)? Theorem \ref{hereditary-minimal} gives a positive answer to this question in the case $N=1$.

\begin{question}
Does there exist a complete characterisation of left hereditary directed $k$-linear categories $\C$?
\end{question}

Here we thus assume $N=1$, but $\C$ is not necessarily the linearisation of a discrete EI category. In this case, one can prove that Conditions (A) and (B) hold, see Lemmas~\ref{condition A holds} and~\ref{condition B holds}. But we could prove that $\C$ is necessarily the free tensor category over some tensor quiver (which is the algebraic analogue of the UFP) only under the (strong) assumption that $\C$ is isomorphic to its associated graded with respect to a certain ideal filtration (this condition holds automatically for the $k$-linearisation of a free EI category), see Remark \ref{associated graded}.

\begin{question}
The alternative approach described in Section \ref{suff2} requires an extra combinatorial condition: the non-existence of left-infinite chains. Can this restriction be removed?
\end{question}

In this paper we mostly consider EI categories satisfying the finite factorisation property FFP, cf.\ Definition~\ref{infinite-compositions}. But there are many interesting examples such that this assumption fails. We thus ask:

\begin{question}
To what extent can the main results in this paper be generalised to EI categories without the FFP?
\end{question}

At the moment we only have some necessary combinatorial conditions, see Lemma~\ref{Theta-totally-ordered} and Proposition~\ref{Theta-well-ordered}, and an additional algebraic condition of which we don't even know whether it can occur in any example, see Remark~\ref{splitting-off-infinite-direct-sum}. A status report is given in Section \ref{without-FFP}.

\subsection{Organisation}

The paper is organised as follows.
\begin{itemize}
    \item Section \ref{preliminaries} recalls  some preliminary knowledge on EI categories and $k$-linear categories as well as their representation theory, and introduces in detail the conditions $(\mathrm{A}_d)$, $(\mathrm{B}_d)$, (A) and (B).
    \item In Section \ref{necessity} we adapt the technique used in \cite{li} to prove the necessity direction of Theorem \ref{hereditary-minimal}.
    \item In Section \ref{sufficiency}, we prove Theorem \ref{bound gldim} and the sufficiency direction of Theorem \ref{hereditary-minimal} by adapting the machinery from \cite{cuntz-quillen} to rings with approximate unit.
    \item Section \ref{applications} discusses the applications to transporter categories, orbit categories, and Quillen categories.
    \item Section \ref{suff2} explains the alternative way to prove Theorem \ref{bound gldim}, additionally leading to Theorem \ref{intro-sum-of-rep}.
    \item Section \ref{without-FFP} treats categories without the FFP, discussing some examples and deriving a necessary condition which can be seen as a substitute for the UFP.
    \item The paper has an appendix describing combinatorial and geometric characterisations of the finite normaliser condition occurring in Theorem~\ref{orbit-hereditary}, in terms of group actions on trees.
\end{itemize}

\section{Preliminaries}
\label{preliminaries}

In this section we describe some preliminary knowledge on the representation theory of categories. Throughout this paper, let $k$ be a commutative unital ring, and let $\C$ be a small category. We often assume in proofs that $\C$ is skeletal, but formulate our results without this assumption. We say that $\C$ is \emph{discrete} when it is a usual category, and $\C$ is \emph{$k$-linear} if it is enriched in $k \Mod$, the category of left $k$-Modules. A set $X$ is called \textit{$k^\times$-finite} if its cardinality is finite and invertible in $k$.

\subsection{Discrete EI categories} \label{discrete ei categories}

In this subsection, let $\C$ be a discrete EI category, which can be thought of as a combination of a group and a preordered set. Explicitly, for each object $c$ in $\C$, the morphism set $\C(c,c)$ is a group which we denote by $G_c$. We can define a preorder $\leqslant$ on the set of objects by letting $c \leqslant d$ if $\C(c, d)$ is nonempty. This preorder induces a partial order $\leqslant$ on the set of isomorphism classes of objects.

Given a pair of objects $c$ and $d$, the morphism set $\C(c,d)$ is a $(G_d,G_c)$-biset, {i.\ e.} a left $(G_d\times G_c^{\mathrm{op}})$-set. As such, it is a disjoint union of transitive bisets
\[G_d\times G_c^{\mathrm{op}}/H_i\]
where $H_i\subseteq G_d\times G_c^{\mathrm{op}}$ is a subgroup. We call all subgroups $H_i$ occuring like this \emph{biset stabilisers in $\C(c,d)$}, and will study the following conditions on the biset stabilisers.

\begin{itemize}
    \item[$(\mathrm{A}_d)$] For all $c < d$, all biset stabilisers in $\C(c,d)$ are locally $k^\times$-finite, in the sense that all their finitely generated subgroups are $k^\times$-finite.
    \item[$(\mathrm{B}_d)$] For all $c < d$, and any biset stabiliser $H_i$ in $\C(c,d)$, $\pr_1(H_i)$ is $k^\times$-finite. Here $\pr_1$ denotes the projection $G_d\times G_c^{\mathrm{op}}\rightarrow G_d$.
\end{itemize}

As mentioned in the introduction, the first condition is symmetric, but the second one is not.

\begin{exam}\label{non-symmetric}
Let $G$ be an infinite, locally $k^{\times}$-finite group, for instance $k=\mathbb{Q}$ and $G=\mathbb{Q}/\mathbb{Z}$, and let $\C$ be a category with two objects $c$ and $d$ such that $G_c = G = G_d$, and $\C(c, d)$ as a biset is generated by an element $\alpha$ such that $G_d$ acts on $\C(c, d)$ freely and every element in $G_c$ fixes every morphism in $\C(c, d)$, and $\C(c, d) = \emptyset$. Then we have $H_i=\{1\}\times G_c^{\mathrm{op}}$ for the only biset stabiliser, and $\C$ satisfies $(\mathrm{A}_d)$ and $(\mathrm{B}_d)$, whereas the opposite category only satisfies $(\mathrm{A}_d)$.
\end{exam}

We now consider factorisation properties of morphisms in $\C$. Recall that a non-invertible morphism in $\C$ is called \emph{unfactorisable} if it can not be written as the composition of two non-invertible morphisms, which is similar to an irreducible element in a ring. In this paper, we mainly consider discrete categories such that every non-invertible morphism can be written as a finite composition of unfactorisable morphisms, except in Section \ref{without-FFP}. We give a special name to this property.

\begin{defi}\label{infinite-compositions}
An EI category $\C$ is said to have the finite factorisation property FFP if every non-invertible morphism is a composition of finitely many unfactorisables.
\end{defi}

\begin{exam}
For an arbitrary EI category $\C$, not every non-invertible morphism can be written as a finite composition of unfactorisable morphisms, or even worse, unfactorisable morphisms do not exist. For example, the poset $\mathbb{R}$ with the usual ordering can be viewed as a category, and in this category every morphism is either invertible or factorisable. Another example is the poset $\mathbb{N} \cup \{ \infty \}$ with the usual ordering. The reader can see that unfactorisable morphisms exist, but the unique morphism from object 1 to object $\infty$ cannot be expressed as a composition of unfactorisable morphisms. We shall emphasise that there \emph{do} exist categories with hereditary category algebra, but without the FFP, see Example \ref{hereditary-no-FFP}.
\end{exam}

For an EI category $\C$ with finite factorisation property, the way to decompose a non-invertible morphism into unfactorisable morphisms is in general not unique. We now recall the unique factorisation property UFP due to the second author \cite[Def.~2.7]{li}. Here we take a slightly altered version from \cite[Def.~6.5.2]{SW} which is appropriate for arbitrary, not necessarily skeletal, EI categories. Loosely speaking, the UFP means that the factorisation of every non-invertible morphism into unfactorisable morphisms is unique up to insertion of automorphisms of objects. Therefore, categories having the UFP are analogous to unique factorisation domains in commutative algebra.

\begin{defi}\label{UFP}
The category $\C$ satisfies the \emph{unique factorisation property (UFP)} if for any two chains
\[x = x_0 \xrightarrow{\alpha_1} x_1 \xrightarrow{\alpha_2} \ldots \xrightarrow{\alpha_n} x_n=y\]
and
\[x = x'_0 \xrightarrow{\alpha'_1} x'_1 \xrightarrow{\alpha'_2} \ldots \xrightarrow{\alpha'_{n'}} x'_{n'}=y\]
of unfactorisable morphisms $\alpha_i$ and $\alpha'_i$ which have the same composition $f\colon x\rightarrow y$, we have $n=n'$ and there are isomorphisms $h_i\colon x_i\rightarrow x'_i$ for $1\le i \le n-1$ such  that
\[h_1 \alpha_1 = \alpha'_1, \quad \alpha'_n h_{n-1} = \alpha_n \quad\mbox{and}\quad \alpha'_i h_{i-1}=h_{i}\alpha_i\quad\mbox{for $2\le i \le n-1$}\,,\]
i.\ e., the following ladder diagram commutes:
\begin{equation*}
\begin{tikzcd}[row sep=large]
x \arrow[r, "\alpha_1"] \arrow[d, "\mathrm{id}_x"] & x_1 \arrow[r, "\alpha_2"] \arrow[d, "h_1"] & x_2 \arrow[d, "h_2"] \arrow[r, "\alpha_3"] & \ldots \arrow[r,"\alpha_{n-1}"] & x_{n-1} \arrow[r, "\alpha_n"] \arrow[d,"h_{n-1}"] & \phantom{\,.}y\phantom{\,.} \arrow[d,"\mathrm{id}_y"]\\
x \arrow[r, "\alpha'_1"] & x'_1 \arrow[r, "\alpha'_2"] & x'_2 \arrow[r,"\alpha'_3"] & \ldots \arrow[r,"\alpha'_{n-1}"] & x'_{n-1} \arrow[r,"\alpha'_n"] & \phantom{\,.}y\,.
\end{tikzcd}
\end{equation*}
\end{defi}

\begin{rema}
The UFP is inherited to full subcategories, as is proved in \cite[Prop.~2.10]{li} -- the proof carries over to infinite categories. Contrarily, one can construct relatively straightforward examples showing that the FFP \emph{is not} inherited to full subcategories.
\end{rema}

There is a concrete combinatorial description for skeletal EI categories with unique factorisation property as follows. Recall that an \emph{EI quiver} $Q$ is a datum $(Q_0, Q_1, s, t, f, g)$ such that $Q_0$ and $Q_1$ are the sets of vertices and arrows respectively, $s$ is the source map, and $t$ is the target map. The map $g$ assign a group to each vertex in $Q_0$, and the map $f$ assigns a $(G_d, G_c)$-biset to each arrow $c \to d$ in $Q_1$. Suppose that the category quiver $Q$ has no loops or oriented cycle and satisfies the finite factorisation property; that is, every direct path $\gamma$ in $Q$ can be written as a finite composition of arrows $\alpha_1 \circ \ldots \circ \alpha_n$. Then we can define an EI category $\C_Q$ as follows:
\begin{itemize}
\item objects in $\C_Q$ are the same as vertices in $Q_0$;
\item for each object $c$, let $G_c = \C(c, c) = f(c)$;
\item for two distinct objects $c$ and $d$, let $\mathcal{P}$ be the set of directed paths from $c$ to $d$, and define
\[
\C(c, d) = \bigsqcup_{\gamma \in \mathcal{P}} g(\alpha_1) \times_{f(c_2)} g(\alpha_2) \times_{f(c_3)} \ldots \times_{f(c_n)} g(\alpha_n)
\]
where $\gamma$ is written as
\[
\xymatrix{
c = c_1 \ar[r]^-{\alpha_1} & c_2 \ar[r]^-{\alpha_2} & \ldots \ar[r]^-{\alpha_{n-1}} & c_n \ar[r]^-{\alpha_{n}} & c_{n+1} = d
}
\]
such that all $\alpha_i$'s are in $Q_1$, and $\times_{f(c_i)}$ denotes the fiber product.
\end{itemize}
The reader can easily see that $\C_Q$ is an EI category. Furthermore, it is not hard to verify the following result: a skeletal EI category $\C$ (with finite factorisation property) has unique factorisation property if and only if it is isomorphic to $\C_Q$ for a certain EI quiver $Q$ (with finite factorisation property). For details, please refer to \cite[Proposition 2.8]{li}. These EI categories are called \emph{free EI categories}.

Let us turn to the representation theory of EI categories, for which the reader can refer to \cite{mitchell, trafo, xu}. A \emph{representation} $V$ of $\C$, or a left \emph{$\C$-module}, is a covariant functor from $\C$ to $k \Mod$. We denote the value of $V$ on an object $c$ in $\C$ by $V(c)$. Morphisms between two $\C$-modules are natural transformations. We denote the category of all left $\C$-modules by $\C \Mod$, which is abelian and has enough projective objects. In particular, the $k$-linearisations of representable functors are projective, and every $\C$-module is a quotient of a \emph{free module}, which is a direct sum of $k$-linearisations of representable functors. Thus we can define projective dimensions for objects in $\C \Mod$, and set the left global dimension $\gldim \C$ of $\C$ over $k$ to be the supremum of the projective dimensions of all left $\C$-modules. If $\gldim \C$ is at most 1, we say that $\C$ is {left hereditary} over $k$. Clearly, $\C$ is left hereditary over $k$ if and only if the category $\C \Mod$ is a \emph{hereditary category}; that is, every submodule of a projective $\C$-module is projective as well.

An EI category $\C$ gives rise to a \emph{category algebra} $k\C$ defined as follows: as a $k$-module it has a basis consisting of all morphisms in $\C$, and its multiplication is defined via the composition of morphisms (in particular, the product of two non-composable morphisms is 0); see \cite[Section 2]{webb}. This is an associative algebra equipped with an \emph{approximate unit} (see Definition \ref{approximate unit}), but in general is not unital. The left global dimension of $\C$ over $k$ defined above may be interpreted as the left global dimension of $k\C$ in a suitable sense, see Subsection \ref{approximate basics}.

In the rest of the paper, the words `module', `global dimension'  or `hereditary' always mean `left module', `left global dimension' and `left hereditary', unless explicitly mentioned otherwise.

\subsection{$k$-linear categories} \label{k-linear categories}

In this subsection, let $\C$ be a \emph{directed} $k$-linear category; that is, if $\C(c,d)\neq 0$ and $\C(d,c)\neq 0$, then $c$ and $d$ are isomorphic. The directedness of $\C$ implies that we can define a partial order on the set of isomorphism classes in $\C$ by writing $[c]\le [d]$ if $\C(c,d)\neq 0$. As usual, we write $[c]<[d]$ if $[c]\le [d]$ and $[c]\neq [d]$.

The endomorphism $k$-algebra of an object $c$ is denoted $R_c$. If $c$ and $d$ are objects, then $\C(c,d)$ is an $(R_d,R_c)$-bimodule, i.\ e. a left $(R_d\otimes R_c^{\mathrm{op}})$-module, where $\otimes$ without subscript denotes $\otimes_k$. Parallel to the combinatorial conditions $(\mathrm{A}_d)$ and $(\mathrm{B}_d)$ on biset stabilisers of EI categories, we will consider the following two conditions on these bimodules:

\begin{itemize}
    \item[(A)] For all $c < d$, $\C(c,d)$ is flat as a right $R_c$-module.
    \item[(B)] For all $c < d$ and all left $R_c$-modules $M$, the tensor product $\C(c,d)\otimes_{R_c} M$ is projective as a left $R_d$-module.
\end{itemize}

Note that Condition (B) implies in particular that $\C(c,d)$ is projective as a left $R_d$-module (take $M$ to be the left regular representation $R_c$). On the other hand, if $\C(c,d)$ is projective as an $(R_d,R_c)$-bimodule, and $k$ is semisimple, then (A) and (B) are satisfied. The category in Example \ref{non-symmetric} yields an example where (A) and (B) are satisfied, but $\C(c,d)$ is not projective as a bimodule. This is implied by the results of Subsection \ref{translation}, in particular Lemmas~\ref{modules-over-group-ring} and~\ref{B-vs-Bd}.

Slightly differently from EI categories, for a $k$-linear category $\C$, a left \emph{$\C$-module} $V$ is a covariant $k$-linear functor from $\C$ to $k \Mod$; that is, $V$ must respect the $k$-linear structure of $\C$. As for EI categories, one can define \emph{left global dimension} and \emph{left hereditary} for $\C$. These can again be interpreted in terms of the category algebra of $\C$, defined similarly as in the previous subsection.

In this paper we mainly consider a special type of $k$-linear categories, called \emph{free tensor categories}, which are analogues of EI categories satisfying the unique factorisation property \cite[Def.~2.1, Def.~2.2, Prop.~2.8]{li}.

\begin{defi}
A \emph{(directed) $k$-linear tensor quiver} $(X,\mathcal{U})$ consists of
\begin{itemize}
    \item a partially ordered set $X$,
    \item a $k$-algebra $R_x$ for every $x\in X$, and
    \item an $(R_y,R_x)$-bimodule $U(x,y)$ for all $x<y$.
\end{itemize}
\end{defi}

Suppose $(X,\mathcal{U})$ is a $k$-linear tensor quiver. For every chain $\gamma=(x_0, \ldots, x_\ell)$ in $X$, where $x_i < x_{i+1}$, define
\[U(\gamma) = U(x_{\ell-1},x_\ell) \otimes_{R_{x_{\ell-1}}} U(x_{\ell-2},x_{\ell-1}) \otimes_{R_{x_{\ell-2}}} \ldots \otimes_{R_{x_1}} U(x_0,x_1)\,.\]
This is an $(R_{x_\ell},R_{x_0})$-bimodule. In particular, if $\gamma$ consists of two entries $x<y$, then $U(\gamma)=U(x,y)$; if $\gamma$ has a single entry $x$, then $U(\gamma)=R_x$; if $\gamma$ is empty, then $U(\gamma)=0$.

For a chain $\gamma=(x_0, \ldots, x_\ell)$ as above, denote $\ell(\gamma)=\ell$. For two chains $\gamma=(x_0,\ldots, x_\ell)$ and $\delta=(y_1,\ldots, y_m)$ with $x_\ell=y_1$, let $\delta\gamma$ denote the concatenation $\delta\gamma=(x_0,\ldots, x_\ell, y_2,\ldots y_m)$. Then we have a canonical map
\begin{align}\label{composition}U(\delta)\otimes U(\gamma) \rightarrow U(\delta)\otimes_{R_{x_\ell}}U(\gamma) \cong U(\delta\gamma)\,.\end{align}

\begin{defi}
The \emph{free tensor category} associated to a $k$-linear tensor quiver $(X,\mathcal{U})$ is the $k$-linear category $\mathcal{T}_k(X,\mathcal{U})$ with object set $X$ and
\[\mathcal{T}_k(X,\mathcal{U})(x,y)=\bigoplus_{\gamma\colon x\rightarrow y} U(\gamma)\]
where the sum runs over all chains $\gamma=(x, x_1,\ldots, x_{\ell-1},y)$ in $X$. Composition is given by \eqref{composition}.
\end{defi}

\begin{rema}\label{UFP-vs-free-tensor}
The reader can check that free tensor categories are parallel to discrete EI categories with unique factorisation property, which are called \emph{free EI categories} in \cite{li}. Actually, one can show the following statement: the $k$-linearisation of an arbitrary discrete EI category $\C$ is a free tensor category over some directed tensor quiver if and only if $\C$ has the UFP (in which case the quiver can be chosen as the quiver of unfactorisable morphisms). It follows a posteriori that $\C$ has the FFP. We omit the proof of this statement since we only use the 'if' direction which is easy. It is also clear that a free tensor category $\mathcal{T}_k(X,\mathcal{U})$ satisfies conditions (A) and (B) above if and only if all $U(x,y)$ satisfy the corresponding conditions. Furthermore, a free tensor category is an $\mathbb{N}$-graded category.
\end{rema}

\section{Necessity}
\label{necessity}

The main goal of this section is to prove the necessity direction of Theorem \ref{hereditary-minimal}. In the first subsection, $\C$ can be an arbitrary hereditary directed $k$-linear category. We prove some lemmas which prepare for Subsection \ref{nec-EI}, but can be formulated more generally. In the second subsection, we completely deal with discrete EI categories.

\subsection{Algebraic consequences of hereditarity}

All modules in this subsection are left modules unless specified else. To simplify the proofs, we also assume that $\C$ is skeletal. Under this assumption, the preorder $\leqslant$ on the objects becomes a partial order.

Given a $k$-linear subcategory $\mathcal{D}$ of $\C$, the embedding functor $\iota: \mathcal{D} \to \C$ induces a pullback functor $\iota^{\ast}: \C \Module \to \D \Module$ by sending a $\C$-module $V$ to $V \circ \iota$. We call this functor a \emph{restriction} functor, and denote it by $\Res^{\C}_{\D}$. This is an exact functor. It has a left adjoint $\C \otimes_{\D} -$ called \emph{induction}, which is a left Kan extension; and a right adjoint functor $K: \D \Module \to \C \Module$ which is a right Kan extension.

A full subcategory $\D$ of $\C$ is called an \emph{ideal} if for every pair of objects $x \leqslant y$ of $\C$, whenever $y$ is an object of $\D$, then $x$ is also an object of $\D$. In this case, the functor $K: \D \Module \to \C \Module$ has an explicit description as follows: for $V \in \D \Module$, $(KV)(x) = V(x)$ whenever $x \in \Ob \D$, and is 0 otherwise. Therefore, we easily obtain the following facts:

\begin{lemma} \label{restriction preserves projectives}
If $\D$ is an ideal of $\C$, then $\downarrow_{\D}^{\C}$ preserves projective modules.
\end{lemma}

\begin{proof}
 The conclusion follows from the adjunction $(\downarrow_{\D}^{\C}, K)$ as well as the exactness of these two functors.
\end{proof}

From now on, we assume that $\C$ is left hereditary.

\begin{lemma} \label{endomorphism algebras are hereditary}
For every object $x \in \Ob \C$, $R_x$ is left hereditary.
\end{lemma}

\begin{proof}
Without loss of generality we can assume that $x$ is the unique minimal object in $\C$. Indeed, if this is not the case, we can consider the full subcategory $\D$ consisting of objects $y$ satisfying $y \geqslant x$. Then  $x$ is the unique minimal object in the category $\D$ (here the uniqueness follows from the assumption that $\C$ is skeletal). Furthermore, let $V$ be an arbitrary $\D$-module. Then $V$ can be viewed as a $\C$-module supported on the set of objects in $\D$. Since $\C$ is hereditary, there exists a short exact sequence $0 \to Q \to P \to V \to 0$ of $\C$-modules such that both $P$ and $Q$ are projective $\C$-modules. By our definition of $V$, we may assume that both $P$ and $Q$ are supported on the set of objects in $\D$, and thus they are actually projective $\D$-modules. Consequently, $V$ as a $\D$-module has projective dimension at most 1. Therefore, $\D$ is also hereditary, so we may replace the category $\C$ by its full subcategory $\D$.

Now, let $V$ be an arbitrary $R_x$-module, and consider a short exact sequence $0 \to Q \to P \to KV \to 0$ such that both $P$ and $Q$ are projective $\C$-modules. Let $\E$ be the full subcategory consisting of $x$. Applying $\Res_{\E}^C$ we obtain a short exact sequence
\begin{equation*}
0 \to Q \downarrow_{\E}^{\C} = Q(x) \to P\downarrow_{\E}^{\C} = P(x) \to (KV) \downarrow_{\E}^{\C} = V \to 0.
\end{equation*}
By Lemma \ref{restriction preserves projectives}, both $P(x)$ and $Q(x)$ are projective $\E$-modules.
\end{proof}

\begin{lemma} \label{BP holds}
Let $P$ be a projective $\C$-module. Then for every $x \in \Ob \C$, $P(x)$ is a projective left $R_x$-module. In particular, for every pair of objects $x$ and $y$, $\C(x, y)$ is a projective $R_y$-module.
\end{lemma}

\begin{proof}
Consider the submodule $Q = \C \cdot P(x)$ of $P$ generated by $P(x)$. Clearly, $Q$ is also a projective $\C$-module. Again, without loss of generality we can assume that $x$ is the unique minimal object of $\C$, and apply $\downarrow_{\D}^{\C}$ to deduce the conclusion by Lemma \ref{restriction preserves projectives}, where $\D$ is the full subcategory with one object $x$.

The first statement of the Lemma is proved, and the second one follows by an application to the projective $\C$-module $\C\cdot e_x$. \end{proof}

\begin{lemma} \label{describe projective modules}
Let $P$ be a projective $\C$-module generated by its value $P(x)$ on a certain object $x$. Then $P \cong \C \otimes_{R_x} P(x)$.
\end{lemma}

\begin{proof}
By Lemma \ref{BP holds}, $P(x)$ is a projective $R_x$-module, so $\C \otimes_{R_x} P(x)$ is a projective $\C$-module. Since $P = \C \cdot P(x)$, there is a natural surjection $\C \otimes_{R_x} P(x) \to P$ induced by multiplication $\alpha \otimes v \mapsto \alpha \cdot v$. Therefore, $\C \otimes_{R_x} P(x) \cong P \oplus Q$ where $Q$ is a projective $\C$-module whose value on $x$ is 0. But $\C \otimes_{R_x} P(x)$ is generated by its value on $x$, hence so is its direct summand $Q$. Consequently, $Q = 0$ and $P \cong \C \otimes_{R_x} P(x)$.
\end{proof}

The following lemma reduces the proofs of some statements to the simple situation of categories with only two objects. Thus we can verify conditions (A) and (B) in this simple situation since they are local conditions on $\C$.

\begin{lemma} \label{subcategories of two objects are hereditary}
Every full subcategory of $\C$ consisting of two objects is left hereditary.
\end{lemma}

\begin{proof}
Let $\D$ be the full subcategory of $\C$ consisting of objects $x$ and $y$. If there is no nonzero morphisms between $x$ and $y$, then $\D = R_x \oplus R_y$. By Lemma \ref{endomorphism algebras are hereditary}, $\D$ is left hereditary. Otherwise, as we did in the proof of previous lemmas, we assume that $x$ is the unique minimal object in $\C$.

Firstly, we claim that the functor $\Res_{\D}^{\C}$ preserves projective modules. Since every projective $\C$-module is a direct summand of a  free modules of the form $\C e_z$, and the restriction functor commutes with direct sums, it is enough to show that restrictions of these free modules are projective. But
\begin{equation*}
(\C e_x) \Res_{\D}^{\C} = \D e_x
\end{equation*}
and for $z > x$
\begin{equation*}
(\C e_z) \Res_{\D}^{\C} = \C (z, y),
\end{equation*}
which is a projective $R_y$-module by Lemma \ref{BP holds}, and hence a projective $\D$-module. The claim is proved.

For an arbitrary $\D$-module $V$, there exists a short exact sequence $0 \to V(y) \to V \to V(x) \to 0$. Therefore, it suffices to show the projective dimensions of $V(x)$ and $V(y)$ as $\D$-modules do not exceed 1. For $V(y)$, the conclusion is clear. Now consider $V(x)$. Taking a short exact sequence $0 \to P \to Q \to V(x) \to 0$ of $\C$-modules with both $P$ and $Q$ projective, and applying $\Res_{\D}^{\C}$, we get a short exact sequence of $\D$-modules such that the first two terms are projective, and the conclusion follows.
\end{proof}

In the above proof both $V(x)$ and $V(y)$ can be viewed as either $\C$-modules or $\D$-modules. But in general $V$ cannot be viewed as a $\C$-module.

\begin{lemma} \label{condition A holds}
Condition $\mathrm{(A)}$ holds.
\end{lemma}

\begin{proof}
By Lemma \ref{subcategories of two objects are hereditary}, we can assume that $\C$ has only two objects $x < y$. Let $V_x$ be an arbitrary $R_x$-module and let $0 \to U_x \to P_x \to V_x \to 0$ be a short exact sequence of $R_x$-modules such that $P_x$ is a projective $R_x$-module. Define a $\C$-module $V = \C \otimes_{R_x} V_x$. Then we have a short exact sequence of $\C$-modules $0 \to U \to \C \otimes_{R_x} P_x \to \C \otimes_{R_x} V_x \to 0$.

By considering the value of $U$ on $y$, one knows that $U(y) \cong (\C(x, y) \otimes_{R_x} U_x) /K$, where $K = \Tor_1^{R_x}(\C(x, y), V_x)$. Therefore, $U$ is generated by $U_x$. But $U$ is a projective $\C$-module, so $U \cong \C \otimes_{R_x} U_x$ by Lemma \ref{describe projective modules} and $U(y) \cong \C(x, y) \otimes_{R_x} U_x$. Therefore, restricted to the object $y$ one gets a short exact sequence
\begin{equation*}
0 \to \C(x, y) \otimes_{R_x} U_x \to \C(x, y) \otimes_{R_x} P_x \to \C(x, y) \otimes_{R_x} V_x \to 0.
\end{equation*}
This means that $\Tor_1^{R_x}(\C(x, y), V_x) = 0$. But $V_x$ was arbitrary, and consequently the functor $\C(x,y) \otimes_{R_x} -$ is exact. That is, $\C(x, y)$ is a flat right $R_x$-module.
\end{proof}

\begin{lemma} \label{classify projective modules}
Suppose that $\C$ has only two objects $x < y$. Let $P$ be a projective $\C$-module. Then $P = Q \oplus P/Q$ where $Q$ is generated by $P(x)$.
\end{lemma}

\begin{proof}
We have a short exact sequence $0 \to Q \to P \to P/Q \to 0$. It suffices to show that $P/Q$ is projective, or equivalently, $\Ext_{\C}^1 (P/Q, V) = 0$ for any $\C$-module $V$. Since there is a short exact sequence $0 \to V(y) \to V \to V(x) \to 0$, it suffices to show that $\Ext_{\C}^1 (P/Q, V(x)) = 0 = \Ext_{\C}^1 (P/Q, V(y))$. Applying the functor $\Hom_{\C} (-, V(y))$ to the first short exact sequence and noting that $\Hom_{\C} (Q, V(y)) = 0$, we deduce the second identity. To obtain the first identity, we just note that $\Hom_{\C} (P/Q, V(x)) = 0$ since $(P/Q)(x) = 0$ and
\begin{equation*}
\Hom_{\C} (P, V(x)) \cong \Hom_{R_x} (P(x), V(x)) = \Hom_{R_x} (Q(x), V(x)) \cong \Hom_{\C}(Q, V(x)).\qedhere
\end{equation*}
\end{proof}

\begin{lemma} \label{condition B holds}
Condition $\mathrm{(B)}$ holds.
\end{lemma}

\begin{proof}
Again, we can assume that $\C$ has only two objects $x < y$ by Lemma \ref{subcategories of two objects are hereditary}. Let $V_x$ be an arbitrary $R_x$-module, and let $V = \C \otimes_{R_x} V_x$. Then $V(y) = \C(x, y) \otimes_{R_x} V_x$. Let $0 \to W_x \to P_x \to V_x \to 0$ be a short exact sequence of $R_x$-modules such that $P_x$ is a projective $R_x$-module, and let $P = \C \otimes_{R_x} P_x$ and $W = \C \otimes_{R_x} W_x$. Note that the functor $\C \otimes_{R_x} -$ is exact by Lemma \ref{condition A holds}, so we have the following commutative diagram of sequences where all rows and columns are exact:
\begin{equation*}
\xymatrix{
 & & 0 \ar[d] & 0 \ar[d]\\
 & 0 \ar[r] & W \ar[r] \ar[d] & Q \ar[r] \ar[d] & V(y) \ar[r] & 0\\
 & & P \ar[d] \ar@{=}[r] & P \ar[d]\\
0 \ar[r] & V(y) \ar[r] & V \ar[r] \ar[d] & V_x \ar[r] \ar[d] & 0\\
 & & 0 & 0
}
\end{equation*}
Note that $W_x = Q(x)$, and $Q$ is a projective $\C$-module since $\C$ is left hereditary. By Lemma \ref{classify projective modules}, $V(y)$ is a projective $\C$-module, and hence a projective $R_y$-module.
\end{proof}

\subsection{Necessity of the UFP}\label{nec-EI}

In this subsection, we consider discrete EI categories $\C$; that is, they are usual set-enriched categories, as opposed to $k$-linear categories. Let $k\C$ be the category algebra defined previously. Moreover, we suppose that $\C$ is skeletal, has the FFP, and its category algebra $k\C$ is left hereditary.

\begin{lemma} \label{intersection of projectives}
Let $\alpha$ and $\beta$ be two morphisms in $\C$ starting from an object $x$. Then either one of $k\C \alpha$ and $k\C \beta$ is contained in the other, or their intersection is 0.
\end{lemma}

\begin{proof}
Let $u$ and $v$ be the target objects of $\alpha$ and $\beta$ respectively (of course, $u$ and $v$ might be the same). Consider the situation that $\alpha \notin k\C \beta$ and $\beta \notin k\C \alpha$; that is, $\alpha$ (resp., $\beta$) cannot be written as a composition of $\beta$ (resp., $\alpha$) and another morphism. In this case, $G_u \alpha$ and $G_v \beta$ are disjoint; that is, either $u \neq v$, or $u = v$ but $\alpha$ and $\beta$ are in different $G_u$-orbits.

Note that $V = k\C \alpha + k\C \beta$ is a projective module as a submodule of $k\C e_x$. Furthermore, $V/\mathfrak{m}V = kG_u \alpha \oplus kG_v \beta$, where $\mathfrak{m}$ is the two-sided ideal of $k\C$ spanned by all non-invertible morphisms. To see this, note that by the assumption and the definition of $\mathfrak{m}$, the images of both $\alpha$ and $\beta$ are nonzero, so are the images of all elements in $G_u \alpha$ or $G_v \beta$. Therefore, there is a natural surjective $k\C$-module homomorphism
\begin{equation*}
k\C \alpha + k\C \beta \to kG_u \alpha \oplus kG_v \beta.
\end{equation*}
But there is also a natural surjection
\begin{equation*}
k\C \alpha \oplus k\C \beta \to kG_u \alpha \oplus kG_v\beta.
\end{equation*}
By the universal property of projective modules, the first map factors through the second map; that is, the first map is a composition
\begin{equation*}
k\C \alpha + k\C\beta \to k\C \alpha \oplus k\C \beta \to kG_u \alpha \oplus kG_v \beta,
\end{equation*}
and when restricted to the full subcategory consisting of objects $u$ and $v$, the two component maps in this composition become identities. However, since $k\C \alpha \oplus k\C \beta$ is generated by $kG_u \alpha$ and $kG_v \beta$, the first component map in this composition must be surjective. This happens if and only if it is actually an isomorphism, or equivalently, $k\C \alpha \cap k\C \beta = 0$.
\end{proof}

Recall that we assume the FFP for $\C$, so every non-invertible morphism $\alpha: x \to y$ in $\C$ can be written as a finite composition of unfactorisable morphisms. We define the \emph{length} $l(\alpha)$ to be
\begin{equation*}
l(\alpha) = \min \{ n \mid \alpha = \alpha_n \circ \ldots \circ \alpha_1\}
\end{equation*}
where each $\alpha_i$ is unfactorisable. By convention, $l(\alpha) = 0$ if $\alpha$ is invertible. Clearly, $l(\alpha) = 1$ if and only if $\alpha$ is unfactorisable.

The following proposition, as well as Lemma \ref{condition A holds} and Lemma \ref{condition B holds}, forms the necessity direction of Theorem \ref{hereditary-minimal}.

\begin{prop}\label{UFP-uniqueness}
The category $\C$ satisfies the UFP.
\end{prop}

\begin{proof}
Let $\alpha: x \to y$ be a non-invertible morphism in $\C$. We prove the conclusion by an induction on $l = l(\alpha)$. The conclusion is clearly true if $l = 1$. Now suppose that $l > 1$. Let
\begin{equation*}
\xymatrix{
x = x_0 \ar[r]^-{\alpha_1} & x_1 \ar[r]^-{\alpha_2} & \ldots \ar[r]^-{\alpha_n} & x_n = y\\
x = y_0 \ar[r]^-{\beta_1} & y_1 \ar[r]^-{\beta_2} & \ldots \ar[r]^-{\beta_m} & y_m = y
}
\end{equation*}
be two decompositions of $\alpha$ into unfactorisable morphisms. Of course, we can assume that $n = l$ or $m = l$ since we can always compare an arbitrary decomposition to a decomposition with minimal length.

Consider $P^1 = k\C \alpha_1$ and $Q^1 = k\C \beta_1$. Since $\alpha$ lies in both modules, by Lemma \ref{intersection of projectives}, either $P^1 \subseteq Q^1$, or $Q^1 \subseteq P^1$. Without loss of generality we assume that $Q^1 \subseteq P^1$. Consequently, $\beta_1 \in k\C \alpha_1$, so $\beta_1 = g_1 \alpha_1$ where $g_1 \in \C(x_1, y_1)$. But $\beta_1$ is unfactorisable, so $g_1$ must be an isomorphism. Therefore, $x_1 = y_1$ since $\C$ is skeletal, and $P^1 = Q^1$.
Furthermore, if we let $\alpha' = \alpha_n \circ \ldots \circ \alpha_2$ and $\beta' = \beta_n \circ \ldots \circ \beta_2$, then
\begin{equation*}
\alpha' \circ \alpha_1 = \alpha = \beta' \circ \beta_1 = (\beta' g_1) \circ \alpha_1.
\end{equation*}

From the above identity we cannot immediately deduce that $\beta' g_1 = \alpha'$ since the choice of $g_1$ in general is not unique. However, we claim that $\beta' g_1$ and $\alpha'$ lie in the same right $G_{x_1}$-orbit. To show it, consider the projective $\C$-module $k\C \alpha_1$. It is generated by $kG_{x_1} \alpha_1$, and hence by Lemma \ref{describe projective modules}, $k\C \alpha_1 \cong k\C \otimes_{kG_{x_1}} kG_{x_1} \alpha_1$ where the map is the usual multiplication. In particular, by restricting this isomorphism to $y$, we have
\begin{equation}
k\C(x_1, y) \otimes_{kG_{x_1}} kG_{x_1} \alpha_1 \cong k\C(x_1, y) \cdot (kG_{x_1} \alpha_1).
\end{equation}

Let $H$ be the stabiliser of $\alpha_1$ in $G_{x_1}$. Note that $kG_{x_1}\alpha_1 \cong k[G_{x_1}/H]$ as left $kG_{x_1}$-modules. Thus, we can write
\begin{equation*}
(k\C\alpha_1)_y \cong k\C(x_1,y)\otimes_{kG_{x_1}}k[G_{x_1}/H] \cong k[\C(x_1,y)/H]\,.
\end{equation*}
Since $(\beta' g_1) \circ \alpha_1$ and $\alpha' \circ \alpha_1$ define the same element of the left hand side, it follows that  $\beta' g_1$ and $\alpha'$ coincide on the right hand side. So there is an element $h\in H$ such that $\beta' g_1 h=\alpha'$, and the above claim is proved. We thus get a commutative diagram
\begin{equation*}
\xymatrix{
x  \ar[r]^-{\alpha_1} \ar[d]^1 & x_1 \ar[r]^-{\alpha'} \ar[d]^{g_1h} & y \ar[d]^1\\
x  \ar[r]^-{\beta_1} & x_1 \ar[r]^-{\beta'} & y
}
\end{equation*}

Now we turn to the morphism $\alpha': x_1 \to y$. Clearly, $l(\alpha') < l(\alpha)$. By the induction hypothesis, it satisfies the UFP. In particular, for the two different decompositions
\begin{equation*}
\xymatrix{
x_1 \ar[r]^-{\alpha_2} & x_2 \ar[r]^-{\alpha_3} & \ldots \ar[r]^-{\alpha_n} & x_n = y\\
x_1 \ar[r]^-{\beta_2g_1 h} & y_2 \ar[r]^-{\beta_3} & \ldots \ar[r]^-{\beta_m} & y_m = y
}
\end{equation*}
we must have $m =n$, and $x_i = y_i$ for $2 \leqslant i \leqslant n$. Furthermore, there exist elements $g_i \in G_{x_i}$ such that the following diagram commutes:
\begin{equation*}
\xymatrix{
x_1 \ar[r]^-{\alpha_2} \ar[d]^1 & x_2 \ar[r]^-{\alpha_3} \ar[d]^{g_2} & \ldots \ar[r]^-{\alpha_n} \ar[d]^{\ldots} & x_n = y \ar[d]^1 \\
x_1 \ar[r]^-{\beta_2g_1 h} & x_2 \ar[r]^-{\beta_3} & \ldots \ar[r]^-{\beta_n} & x_n = y
}
\end{equation*}
By gluing the two commutative diagrams, we obtain the following commutative diagram
\begin{equation*}
\xymatrix{
x \ar[r]^-{\alpha_1} \ar[d]^1 & x_1 \ar[r]^-{\alpha_2} \ar[d]^{g_1h} & x_2 \ar[r]^-{\alpha_3} \ar[d]^{g_2} & \ldots \ar[r]^-{\alpha_n} \ar[d]^{\ldots} & x_n = y \ar[d]^1 \\
x \ar[r]^-{\beta_1} & x_1 \ar[r]^-{\beta_2} & x_2 \ar[r]^-{\beta_3} & \ldots \ar[r]^-{\beta_n} & x_n = y
}
\end{equation*}
That is, $\alpha$ satisfies the UFP. The conclusion follows by induction.
\end{proof}

\begin{rema} \label{associated graded}
By the above proposition and Remark \ref{UFP-vs-free-tensor}, we conclude that if the category algebra of a discrete EI category $\C$ with the FFP is left hereditary, then the $k$-linearisation of $\C$ is a free tensor category over some directed tensor quiver. This result can be extended to the $k$-linear situation. Explicitly, let $\C$ be a skeletal directed $k$-linear category, and let $\mathfrak{m}$ be the direct sum of those $\C(c, d)$ with $c \neq d$, which is a two-sided ideal of $\C$. Then we can define an associated $\mathbb{N}$-graded category $\widehat{\C}$ such that:
\begin{itemize}
\item $\widehat{\C}$ and $\C$ have the same objects;
\item for any object $c$, $\widehat{\C}(c, c)_0 = \C(c, c)$, and $\widehat{\C}(c, c)_n = 0$ for $n \geq 1$;
\item for $c \neq d$, $\widehat{\C}(c, d)_0 = 0$, and $\widehat{\C}(c, d)_n = e_d (\mathfrak{m}^n / \mathfrak{m}^{n+1}) e_c$ for $n \geq 1$.
\end{itemize}
Then by slightly changing the above argumentation, we can show that if $\widehat{\C}$ (forgetting its graded structure) is left hereditary, then $\widehat{\C}$ is isomorphic to a free tensor category $\mathcal{T}_k (\mathcal{X}, \mathcal{U})$ over a tensor quiver $(\mathcal{X}, \mathcal{U})$ where $\mathcal{X}$ is the set of all objects in $\C$ and $U(x, y)$ is spanned by the set of unfactorisable morphisms from $x$ to $y$ for each pair of objects $x$ and $y$. However, this leaves an interesting question as follows: if $\C$ is left hereditary , is it always true that $\C$ is isomorphic to $\widehat{\C}$, and hence is isomorphic to a free tensor category?
\end{rema}

\section{Sufficiency} \label{sufficiency}

The main goal of this section is to prove Theorem \ref{bound gldim}, relying on techniques of  \cite{cuntz-quillen}, and to deduce the sufficiency direction of Theorem \ref{hereditary-minimal} by a translation process between discrete EI categories and directed $k$-linear categories.

\subsection{Basic facts about rings with approximate unit}
\label{approximate basics}

We recall here some basic facts about rings with approximate unit, paralleling \cite[Sec.~6.2]{SW}, and adapt some standard results to this setting. Recall that a \emph{net} in a set $S$ is a map $I\rightarrow S$ where $I$ is a directed set, i.\ e. a partially ordered set in which any two elements  have a common upper bound.

\begin{defi}
\label{approximate unit}
A ring $S$ has an \emph{approximate unit} if there is a net $(e_i)_{i \in I}$ of idempotents in $S$ with the following two properties:
\begin{itemize}
\item For every $s\in S$, there is some $i$ such that $e_is=s=se_i$.
\item For $i\le j$, we have $e_j e_i = e_i e_j= e_i$.
\end{itemize}
A left $S$-module $M$ is called \emph{non-degenerate} if $SM=M$. Equivalently, if for every $m\in M$ there is some $i$ such that $e_im=m$. An $(S,S)$-bimodule is called \emph{non-degenerate} if it is non-degenerate as a left and as a right $S$-module.
\end{defi}

The category of non-degenerate left $S$-modules is an abelian category and thus has a meaningful notion of projective dimension and global dimension \cite[Sec.~9]{mitchell}. We refer to this dimension if we talk about the global dimension of a ring with approximate unit. If $e_i$ is idempotent, then $Se_i$ is projective.  It follows that the category of non-degenerate $S$-modules has enough projectives, so that it is hereditary if and only if submodules of projectives are projective.

If $\C$ is a $k$-linear category, then the category algebra $k\C$ has an approximate unit given by sums of the form $\sum\limits_{x\in F} e_x$, where $F$ runs through finite subsets of $\mathrm{Ob}(\C)$, ordered by inclusion, and $e_x=\mathrm{id}_x$ as usual. There is an equivalence of $k$-linear categories between the category of non-degenerate $k\C$-modules and the category of representations of $\C$, i.~e.\ $k$-linear functors from $\C$ to $k \Mod$; see \cite[Thm.~7.1]{mitchell} and \cite[Prop.~6.2.4]{SW}.

\begin{rema} \label{approximately Noetherian}
By a result of Auslander \cite{auslander}, the left and right global dimensions of a Noetherian ring coincide. The proof can easily be adapted to rings with approximate unit via replacing the Noetherian condition by the \emph{approximately Noetherian} condition that for every $i$, every left $S$-submodule of $Se_i$ and every right $S$-submodule of $e_iS$ is finitely generated as a left, resp.\ right $S$-module. In particular, using Theorem \ref{hereditary-minimal}, the category algebra of the category $\C$ introduced in Example~\ref{non-symmetric} is \emph{not} approximately Noetherian.
\end{rema}

As in the unital case, projective non-degenerate modules are flat:

\begin{citedlemma}[{\cite[Lemma~6.2.3]{SW}}]
\label{tensor}
Let $S$ be a ring with  approximate unit.\\
(a)  If $M$ is a non-degenerate left $S$-module, then there is a natural isomorphism of $S$-modules
\[S\otimes_S M \cong M\,.\]
(b) A non-degenerate left $S$-module $P$ which is projective in the category of non-degenerate left $S$-modules is flat in the sense that $-\otimes_S P$ is an exact from non-degenerate $S$-modules to abelian groups.
\end{citedlemma}

\begin{rema}\label{tor}
Part (b) implies that one can define $\mathrm{Tor}$ terms in the usual way via projective resolutions, which are symmetric and yield long exact $\mathrm{Tor}$ sequences for every short exact sequence of non-degenerate $S$-modules.
\end{rema}

Let $A$ be an $S$-algebra with approximate unit, i.~e.\ $A$ is a ring equipped with a ring homomorphism $S\rightarrow A$ such that the image of the approximate unit of $S$ constitutes an approximate unit of $A$.
The following definitions and results are taken from \cite[Sec.~2]{cuntz-quillen}, where they are proved for unital rings, and adapted to the case of rings with approximate unit. See also \cite{hazewinkel-gubareni-kirichenko}.

\begin{defi}
The \emph{$(A,A)$-bimodule of differential forms of degree one} is defined as $\Omega_S^1 A=A\otimes_S A/S$ with the 'Leibniz' $A$-bimodule structure
\[a\cdot (b\otimes [c])\cdot d = ab\otimes [cd] - abc\otimes [d]\,.\]
\end{defi}

This is easily checked to be a non-degenerate $(A,A)$-bimodule.

\begin{defi}
Let $S$ be a ring with approximate unit, $A$ an $S$-algebra, and $M$ an $(A,A)$-bimodule. The abelian group of \emph{$S$-derivations} $\mathrm{Der}_{S}(A,M)$ consists of all derivations $D\colon A\rightarrow M$ with $DS=0$.
\end{defi}

There is a canonical $S$-derivation $d\colon A\rightarrow \Omega_S^1 A$ sending $a$ to $e_i\otimes [a]$, where $e_i\in S$ is chosen such that $ae_i=e_ia=a$. This is well-defined: If $e_j$ is another such element, we may assume $i\le j$ and get
\[e_i\otimes [a] = e_je_i\otimes [a] = e_j\otimes [e_ia]=e_j\otimes [a]\,.\]
It is a derivation by the definition of the Leibniz bimodule structure.

\begin{lemma}\label{omega-universal}
$d$ is a universal $S$-derivation, furnishing an isomorphism
\[\mathrm{Hom}_{(A,A)}(\Omega_S^1 A, M) \cong \mathrm{Der}_S(A, M)\,.\]
\end{lemma}

\begin{proof}
We have to show that for an $S$-derivation $D\colon A\rightarrow M$, there is a unique $(A,A)$-linear map $F\colon \Omega_S^1 A\rightarrow M$ such that $F\circ d=D$, i.\ e. $F(e_i\otimes [a])=Da$ with $e_i$ as above. Uniqueness is clear: Let $a, b\in A$ and choose $e_i$ such that $ae_i=e_ia=a$ and $be_i=e_ib=b$. Then
\[F(a\otimes [b]) = F(a\cdot (e_i\otimes [b])) = a\cdot F(e_i\otimes [b]) = a\cdot Db\,.\]
On the other hand, one easily checks as above that this furnishes a well-defined $(A,A)$-bilinear map $F\colon \Omega_S^1 A \rightarrow M$ which satisfies $F\circ d=D$, so that we have proved existence.
\end{proof}

\begin{prop} \label{cuntz-quillen} Let $S$ be a ring with approximate unit and $A$ an $S$-algebra with approximate unit. There is a short exact sequence
\begin{equation*}
0\longrightarrow \Omega_S^1 A \xrightarrow{\phantom{a}\kappa\phantom{a}} A\otimes_S A \xrightarrow{\phantom{a}m\phantom{a}} A \longrightarrow 0\end{equation*}
of $(A,A)$-bimodules, with $m(a_0\otimes a_1) = a_0a_1$ and $\kappa$ defined in the proof below.
\end{prop}

\begin{proof} We define an $(A,S)$-linear section $\iota$ of the multiplication map $m$ as follows: For $a\in A$, choose $e_i$ with $e_ia=ae_i=a$ and set $\iota(a)=a\otimes e_i$. This is well-defined: Suppose $e_j$ is another such element. Since $I$ is directed, we may assume that $i\le j$. Then
\[a\otimes e_i = a\otimes e_ie_j = a e_i \otimes e_j= a \otimes e_j\,.\]
Thus, the kernel of $m$ is identified (as $(A,S)$-bimodule) with the cokernel of $\iota$ via the projection $\mathrm{id}_{A\otimes_S  A}-\iota m$. Since $\iota$ is given by the canonical isomorphism
\[A\cong S\otimes_S A \]
stemming from the fact that $A$ is a non-degenerate module over $S$, cf. Lemma~\ref{tensor} (a), followed by the canonical morphism $S\rightarrow A$, and tensoring is right exact (over an arbitrary non-unital ring), we identify the cokernel of $\iota$ with $A\otimes_S A/S=\Omega_S^1 A$. It follows that the map
\[\kappa\colon \Omega_S^1 A \rightarrow A \otimes_S A \]
sending $a\otimes [b]$ to
\[(\mathrm{id}-\iota m)(a\otimes b) = a\otimes b - ab\otimes e_i\,,\]
where $e_i$ is chosen such that $e_iab=ab=abe_i$, is well-defined and renders the above sequence exact. Finally, one checks that $\kappa$ is a morphism of $(A,A)$-bimodules if $\Omega_S^1 A$ has the Leibniz bimodule structure.
\end{proof}

\subsection{Cuntz-Quillen proof of sufficiency}

In this section, $\C=\mathcal{T}_k(X,\mathcal{U})$ is the free tensor category over a $k$-linear tensor quiver $(X,\mathcal{U})$ with $R_x$ hereditary for every $x\in X$ which satisfies Conditions (A) and (B) as introduced in Subsection \ref{k-linear categories}. We will now prove that the category algebra of $\C$ over $k$ is hereditary. Let $R$ denote the ring $\bigoplus\limits_{c} R_c$. It is a hereditary ring with approximate unit.

Proposition \ref{cuntz-quillen} gives us a short exact sequence
\begin{equation} \label{cuntz-quillen-ses}0\longrightarrow \Omega_R^1 \C \xrightarrow{\phantom{a}\kappa\phantom{a}} \C\otimes_R \C \xrightarrow{\phantom{a}m\phantom{a}} \C \longrightarrow 0\end{equation}
of $(\C,\C)$-bimodules.

\begin{lemma}
Let $M$ be any left $\C$-module. There is an exact sequence of left $\C$-modules
\begin{equation} \label{ses-M} 0\longrightarrow \Omega_R^1 \C\otimes_{\C} M \longrightarrow \C\otimes_R M \longrightarrow M\longrightarrow 0\,.\end{equation}
\end{lemma}

\begin{proof}
Tensor the exact sequence \eqref{cuntz-quillen-ses} from the right with $M$ and note that it stays exact since the last term $\C$ is a flat right $\C$-module and thus $\mathrm{Tor}_\C^1(\C,M)=0$. This uses Remark \ref{tor}. \end{proof}

\begin{lemma} \label{hom-dim-CM} Let $N$ be a left $R$-module. Then the projective dimension of $\C\otimes_R N$ as a left $\C$-module is at most the projective dimension of $N$ over $R$ (which is at most $1$).
\end{lemma}

\begin{proof}
Condition (A) means that $\C$ is flat over $R$. Take a projective resolution of $N$ over $R$, tensor it up to $\C$ and it will stay exact.
\end{proof}

Let $U=\bigoplus_{x<y} (x,y)$ denote the $(R,R)$-bimodule of unfactorisables.

\begin{prop} \label{omega-of-tensor}
There is an isomorphism $\Omega_1^R\C\cong \C\otimes_R U\otimes_R \C$ of $(\C,\C)$-bimodules.
\end{prop}

\begin{proof}
The proof is the same as in \cite[Prop. 2.6]{cuntz-quillen}, using Lemma~\ref{omega-universal}.
\end{proof}

Now we are ready to prove the main theorem of this section.

\begin{thm} \label{stronger result}
Let $(X,\mathcal{U})$ be a tensor quiver with hereditary $R_x$ for every $x\in X$, satisfying conditions (A) and (B). Then $\mathcal{T}_k(X,\mathcal{U})$ is hereditary. \end{thm}

\begin{proof}
Let $M$ be an arbitrary left $\C$-module. Consider the short exact sequence \eqref{ses-M}. The first term
\[\Omega_1^R \C\otimes_\C M\cong \C\otimes_R (U\otimes_R M)\]
is a projective left $\C$-module since $U\otimes_R M$ is a projective left $R$-module by (B). It follows from \cite[Lemma 9.1]{mitchell} that
\[\mathrm{p.dim}_\C(M) \le \mathrm{max}(\mathrm{p.dim}_\C(\C\otimes_R M), 1)\]
and this implies the claim by Lemma \ref{hom-dim-CM}.
\end{proof}

\subsection{Translation between the $k$-linear and the discrete case}
\label{translation}

In this subsection we discover some close relations between previously established results for discrete EI categories and for $k$-linear categories. In particular, we show that the conditions ($\mathrm{A}_d$) and ($\mathrm{B}_d$) for discrete categories and the conditions (A) and (B) for $k$-linear categories are equivalent under certain moderate assumptions, and hence complete the proof of Theorem \ref{hereditary-minimal}. In this subsection, let $k$ be a semisimple ring.

Part (a) of the following lemma is well-known.

\begin{lemma}\label{modules-over-group-ring}
Let $X$ be a left $G$-set.\\
(a) $kX$ is a projective left $kG$-module if and only if all stabilisers occuring in $X$ are $k^\times$-finite. \\
(b) $kX$ is a flat left $kG$-module if and only if all stabilisers occuring in $X$ are locally $k^\times$-finite, i.\ e., all their finitely generated subgroups are $k^\times$-finite.
\end{lemma}

\begin{proof}
(a) Every $G$-set is a disjoint union of transitive $G$-sets, and a direct sum is projective if and only if each summand is. We may thus assume that $X=G/H$ for some subgroup $H$, and have to show that $k(G/H)$ is projective if and only if $H$ is $k^\times$-finite. For this, consider the canonical surjection
\[\pi\colon kG \twoheadrightarrow k(G/H)\,.\]
It has a section $s$ if and only if $k(G/H)$ is projective. If $s$ exists, then $s([1])$ has equal entries in all left $H$-cosets, hence $H$ is finite. Consequently, $\pi(s([1])$ is divisible by $|H|$ and it follows that $H$ is $k^\times$-finite. Finally, if $H$ is $k$-finite, then a section $s$ can be defined by
\[s([gH])=\frac 1{|H|}\sum_{h\in H} gh\,.\]
(b) As in the proof of (a), we may assume that $X=G/H$. Suppose that $H$ is locally $k^\times$-finite. Then $H$ is a filtered union of $k^\times$-finite subgroups $H_i$, thus $G/H$ is a filtered colimit of $G/H_i$ and $k(G/H)$ is a filtered colimit of projectives $k(G/H_i)$ and hence flat.

Now, suppose that $k(G/H)$ is flat. For $h \in H$, consider the map of right $kG$-modules $kG \rightarrow kG$ given by left multiplication $\lambda_{h-1}$ with $h-1$. It induces a non-injective map after tensoring with $k(G/H)$. Since $k(G/H)$ is flat, the original map has to have nontrivial kernel. Let $x=(x_g g)_{g\in G}$ be nonzero in the kernel. Let $F$ be the finite, nonempty set of elements $g$ with $x_g\neq 0$. Since $x=hx$, $F$ is invariant under left multiplication with $h$, i.\ e. the subgroup generated by $h$ acts on $F$. This action is free since $G$ is a group. It follows that $h$ has finite order, which we denote by $m$. Set $N(h)=1+h+h^2+...+h^{m-1} \in kG$. It is easily checked that
\[\mathrm{ker}(\lambda_{h-1})=\mathrm{im}(\lambda_{N(h)})\,.\]
By exactness, this remains true after tensoring with $k(G/H)$. But then $1H$ lies in the kernel, and it follows that there exist finitely many $k_i$ and $g_i$ such that
\[1H= N(h)\sum_{i=1}^n k_i g_iH = \sum_{i=1}^n\sum_{j=1}^m k_i h^j g_iH\,.\]
In the above double sum, we now focus on those summands  with $h^jg_iH=1H$. (All other summands cancel each other out.) These satisfy $h^jg_i\in H$ and thus $g_i\in H$. It follows that also all the other $h^{j'}g_i$ lie in $H$ and we have
\[1H = \sum_{\substack{i=1\\g_i\in H}}^n mk_i g_i H = \left( m \sum_{\substack{i=1\\g_i\in H}}^n k_i \right)\cdot 1H\,.\]
Thus, $m$ is invertible in $k$.

For several elements $h_1, \ldots, h_n$, consider similarly the map
\[kG \longrightarrow \bigoplus_{i=1}^n kG\] where the $i$-th component is given by left multiplication with $h_i-1$. Again, we find a nonzero $x$ in the kernel, and this time, all $h_i$ have to stabilise $F$ under right multiplication, i.\ e. $\langle h_1, \ldots, h_n\rangle$ acts on $F$ freely and thus it is a finite group. To show that it is $k^\times$-finite, run the same argument as above with the sum of all elements of the subgroup $\langle h_1, \ldots, h_n\rangle$ as norm element.
\end{proof}

\begin{lemma}\label{B-vs-Bd}
If $\C$ is a discrete EI category and $k$ is semisimple, then (B) and $(\mathrm{B}_d)$ are equivalent.
\end{lemma}

\begin{proof}
If (B) holds, take $M$ to be the trivial left $G_c$-module $k=k(G_c/G_c)$. We get that
\[k((G_d\times G_c^{\mathrm{op}})/H_i)\otimes_{kG_c} k(G_c/G_c) \cong k((G_d\times G_c^{\mathrm{op}})/H_i\times_{G_c}G_c/G_c)\cong k(G_d/\mathrm{pr}_1(H_i))\]
is a projective left $kG_d$-module, so $\mathrm{pr}_1(H_i)$ is $k^\times$-finite by Lemma \ref{modules-over-group-ring} (a).

Conversely, suppose that $(\mathrm{B}_d)$ holds. Let $M$ be an arbitrary left $kG_c$-module and $H_i$ a biset stabiliser for $\C(c,d)$. Let $L=\mathrm{pr}_1(H_i)$. Then $H_i$ is contained in $L\times G_c^{\mathrm{op}}$, so $(G_d\times G_c^{\mathrm{op}})/H_i$ can be written as $G_d\times_L ((L\times G_c^{\mathrm{op}})/H_i)$. It follows that \[\C(c,d)\otimes_{G_c} M \cong kG_d \otimes_{kL} \left[k((L\times G_c^{\mathrm{op}})/H_i)\otimes_{G_c} M\right]\,.\]
The term in square brackets is a left $kL$-module which is projective since $kL$ is semisimple by Maschke's theorem. Inducing it up to $kG_d$ yields a projective left $kG_d$-module.
\end{proof}

Finally, we cite the following result of Dicks, characterising hereditary group rings.

\begin{citedprop}[{\cite[Thm.~1]{dicks}}] \label{dicks} Let $k$ be an arbitrary ring and $G$ a group. Then $kG$ is  hereditary if and only if at least one of the following holds:

\begin{itemize}
    \item[(H1)] $k$ is completely reducible and $G$ is the fundamental group of a connected graph of $k^\times$-finite groups.
    \item[(H2)] $k$ is (left) $\aleph_0$-Noetherian and von Neumann regular, and $G$ is countable and locally $k^\times$-finite.
    \item[(H3)] $k$ is hereditary and $G$ is $k^\times$-finite.
\end{itemize}
\end{citedprop}

The conditions on the ring $k$ are explained on the first page of \cite{dicks}.

\begin{rema}\label{dicks-field}
If $k$ is a field, then all conditions on $k$ hold trivially. Thus, $kG$ is hereditary if and only if $G$ is the fundamental group of a connected graph of $k^\times$-finite groups, or is countable and locally $k^\times$-finite.
\end{rema}

We complete the proof of Theorem \ref{hereditary-minimal} and Corollary \ref{two-sided hereditary} in the following proposition.

\begin{prop}
Let $\C$ be a discrete EI category with the unique factorisation property UFP such that all group rings are hereditary. If Conditions $\mathrm{(A_d)}$ and $\mathrm{(B_d)}$ hold, then the category of left $\C$-modules is hereditary.
\end{prop}

\begin{proof}
Remark \ref{UFP-vs-free-tensor} shows that the $k$-linearsation of $\C$ is a free tensor category over a certain tensor quiver. Part (b) of Lemma \ref{modules-over-group-ring} (respectively, Lemma \ref{B-vs-Bd}) tells us that Condition $\mathrm{(A_d)}$ (resp., Condition $\mathrm{(B_d)}$) for discrete EI categories $\C$ is equivalent to Condition (A) (resp., Condition (B)) for its $k$-linearisation. The conclusion then follows from Theorem \ref{stronger result}.
\end{proof}

\subsection{Bounded global dimension}\label{ANBN}

In this subsection, we assume that the rings $R_x$ at the objects of the tensor quiver $(X,\mathcal{U})$ might not be hereditary, but still have global dimension bounded by some natural number $N\ge 1$. Our argumentation still works in this case, and Conditions (A) and (B) can even slightly be weakened.

\begin{defi}
A module $K$ (in some abelian category) is called an \emph{$N$-th kernel} if there are projectives $P_1, \ldots P_N$ and a short exact sequence
\begin{equation}\label{Nth-kernel}
    0 \longrightarrow K \longrightarrow P_N \longrightarrow \ldots \longrightarrow P_1 \,.
\end{equation}
A $0$-th kernel is an arbitrary module by this definition. We also say that a $(-1)$-st, $(-2)$-nd kernel etc. is an arbitrary module.
\end{defi}

\begin{rema}
An abelian category with enough projectives has global dimension $\le N$ if and only if all $N$-th kernels are projective.
\end{rema}

Fix a number $N$ and consider the following conditions. We formulate these in terms of the $(R,R)$-bimodule $\C = \mathcal{T}_k(\mathcal{X}, \mathcal{U})$, but they could equivalently be formulated for the $(R_d,R_c)$-bimodules $\C(c,d)$, as we did with (A) and (B), or for the quiver bimodules $U(c,d)$.

\begin{itemize}
    \item[$(\mathrm{A}_N)$] For every $(N-1)$-st kernel $L$ in the category of left $R$-modules, we have \[\mathrm{Tor}_{R}^1(\C,L)=0\,.\]
    \item[$(\mathrm{B}_N)$] For every short exact sequence of left $R$-modules \[0\longrightarrow Q \longrightarrow P \longrightarrow L \longrightarrow 0\] where $Q$, $P$ are projective and $L$ is an $(N-1)$-st kernel, the induced sequence of left $R$-modules \[\C\otimes_{R} Q \longrightarrow \C \otimes_{R} P \longrightarrow \C\otimes_{R} L \longrightarrow 0\]
    splits, i.\ e. \[\C\otimes_{R} P \cong \C \otimes_{R} L\oplus \mathrm{im}(\C\otimes_{R} Q)\,.\]
    \item[(BP)] $\C$ is projective as a left $R$-module.
\end{itemize}

\begin{rema}
Let $N=1$. Then it can easily be seen that ($\mathrm{A}_1$) is equivalent to (A), and that the conjunction of ($\mathrm{B}_1$) and (BP) is equivalent to (B).
\end{rema}

Note that if (BP) holds, then $(\mathrm{B}_N)$ is equivalent to the fact that $\C\otimes_R L$ is a projective left $R$-module, i.\ e.\ that $\C(c,d)\otimes_{R_c} L$ is a projective left $R_d$-module for all $c<d$. Also, the following lemma is immediate:

\begin{lemma} \label{objectwise-projective} Let $(X,\mathcal{U})$ satisfy (BP).
If $P$ is a projective $k\C$-module, then $P(c)$ is a projective $R_c$-module for every $c\in \mathrm{Ob}(\C)$. If $M$ is of projective dimension at most $m$, then so is $M(c)$.\qed
\end{lemma}

We now prove the analogue of Theorem \ref{stronger result}, which implies Theorem \ref{bound gldim} from the introduction:

\begin{thm} \label{sufficiency-ANBN}
Let $N\ge 1$, and let $(X,\mathcal{U})$ be a tensor quiver with $\gldim R_x\le N$ for all $x\in X$, satisfying conditions $(\mathrm{A}_N)$, $(\mathrm{B}_N)$ and (BP). Then
\[\gldim \mathcal{T}_k(X,\mathcal{U}) \le N\,.\]
\end{thm}

\begin{proof}
Let $K$ be an $N$-th kernel, and let $L$ be the cokernel of $K\rightarrow P_N$, with $P_N$ as in \eqref{Nth-kernel}. Then $L$ is an $(N-1)$-st kernel, and if we show that $L$ has projective dimension at most $1$ as a left $\C$-module, then $K$ is projective by Schanuel's Lemma.

Proposition \ref{cuntz-quillen} and Lemma \ref{ses-M} still hold as above, yielding a short exact sequence
\[ 0 \longrightarrow \Omega_R^1\C \otimes_\C L \longrightarrow \C\otimes_R L \longrightarrow L \longrightarrow 0\,.\]
The middle term has projective dimension $1$: Since the global dimension of $R$ is at most $N$, we know that $L$ has projective dimension at most $1$ as a left $R$-module, so there is a short exact sequence
\[0\longrightarrow Q \longrightarrow P \longrightarrow L\longrightarrow 0\]
of left $R$-modules, and this stays exact after applying $\C\otimes_R -$ by $(\mathrm{A}_N)$. Since Proposition \ref{omega-of-tensor} still applies, the left term is isomorphic to $\C\otimes_R (U\otimes_R L)$. Now, $U\otimes_R L$ is a projective left $R$-module by $(\mathrm{B}_N)$ and (BP), so the first term is projective as well. The proof is finished as the proof of Theorem \ref{stronger result} result above.
\end{proof}

\section{Applications}
\label{applications}

In this section, we apply Theorem \ref{hereditary-minimal} to some examples. These applications were the original interest of the authors to write this paper. In particular, we treat three different categories which shed light on the system of finite subgroups of a given group $G$ and its $G$-structure given by the conjugation action: the transporter category of the subgroup poset, the orbit category,  and the Quillen category. We might consider all finite subgroups of $G$, or single out a certain subsystem.

Throughout, $k$ is a field, $G$ is a discrete group, finite or infinite, and $\F$ is a family of finite subgroups in the following sense:

\begin{defi} A family of finite subgroups of $G$ is a set $\F$ of finite subgroups of $G$ which is non-empty, and closed under conjugation as  well as under passage to subgroups.
\end{defi}

\begin{rema}
In particular, according to our definition, every family contains the trivial subgroup $\{1\}$, and this is heavily used in our discussion. In the literature, orbit and Quillen categories are sometimes considered for arbitrary sets $S$ of subgroups instead of a family $\F$. We restricted to families for simplicity, although Theorem \ref{hereditary-minimal} can of course also be applied in the more general case.
\end{rema}

\subsection{Transporter categories}
\label{transporter categories}

Let $G$ be a group and $\P$ a $G$-poset, i.\ e. a poset which has an action of $G$ through poset-automorphisms. The following definition is essentially taken from \cite{xu-transporter}.

\begin{defi}
The transporter category of $G$ on $\P$, $\P\rtimes G$, has object set $\P$ and
\[\mathrm{Hom}_{\P\rtimes G}(x,y) = \mathrm{Trans}_G(x,y)=\{g \in G; gx\le y\}\,.\]
Composition is given by multiplication in $G$.
\end{defi}

It is easily seen that $\P\rtimes G$ is an EI category if and only if the following condition holds:
\begin{equation}\tag{S}\label{S}
\parbox{\dimexpr\linewidth-4em}{$\quad\quad\quad \quad\quad\quad\quad\quad\quad$ If $gx\le x$, then $gx=x$.}
\end{equation}

This is automatic for finite posets, but not in general, consider e.\ g. the addition action of $\mathbb{Z}$ on itself with the standard ordering. However, we will assume from now on that $\P$ satisfies \eqref{S}. In particular,
\[\mathrm{Hom}_{\P\rtimes G}(x,x)=\mathrm{Stab}_G(x)\]
-- this is actually another equivalent formulation of \eqref{S}.

\begin{lemma}
Condition ($\mathrm{B}_d$) holds if and only if
for all $x, y\in \P$ with $x<y$, $\mathrm{Stab}_G(x)\cap \mathrm{Stab}_G(y)$ is $k^\times$-finite. In this case, also Condition ($\mathrm{A}_d$) holds.
\end{lemma}

\begin{proof}
Let $g\in \mathrm{Hom}_{\P\rtimes G}(x,y)$, then $gx< y$. The biset stabiliser of $g$ in $\mathrm{Stab}_G(y)\times \mathrm{Stab}_G(x)^{\mathrm{op}}$ is
\[H = \{(h,k); hgk=g\}=\{(h,k); h=gk^{-1}g^{-1}\}\,.\]
From this one sees that both projections $\mathrm{pr}_i$ are injective, in particular ($\mathrm{B}_d$) implies ($\mathrm{A}_d$), and \[\mathrm{pr}_1(H)=\mathrm{Stab}_G(y)\cap g \mathrm{Stab}_G(x)g^{-1} = \mathrm{Stab}_G(y)\cap  \mathrm{Stab}_G(gx)\,.\qedhere\]
\end{proof}

\begin{defi}
(a) A finite chain $x_1 < x_2 < \ldots < x_n$ is called saturated if for all $1\le i\le n-1$, there is no $z\in \P$ with $x_i<z<x_{i+1}$.\\
(b) A poset $\P$ is said to satisfy existence of saturated chains (ESC) if any two $x,y\in \P$ with $x<y$ can be joined by a saturated chain.\\
(b) A poset $\P$ is said to satisfy  uniqueness of saturated chains (USC) if the following holds: For any $x$ and $y$, there is at most one saturated chain starting in $x$ and ending in $y$.
\end{defi}

\begin{exam}
The following poset does not satisfy USC:
\begin{equation*}
\begin{tikzcd}[row sep=small]
& \bullet \arrow[rd] & \\
\bullet \arrow[ru] \arrow[rd] & & \bullet\\
 & \bullet \arrow[ru] &
\end{tikzcd}
\end{equation*}

The following poset satisfies USC:
\begin{equation*}
\begin{tikzcd}[row sep=large]
& & \bullet  & \\
\bullet \arrow[rru] \arrow[rrd] &  \bullet \arrow[ru]\arrow[rd]&\\
&  & \bullet
\end{tikzcd}
\end{equation*}
\end{exam}

\begin{lemma} \label{ESC,USC} (a) P satisfies ESC if and only if  $\P\rtimes G$ has the FFP.\\
(b) $\P$ satisfies USC if and only if  $\P\rtimes G$ has the UFP.
\end{lemma}

\begin{proof} Consider a chain of morphisms
\[x_0 \xrightarrow{g_1} x_1 \xrightarrow{g_2} x_2 \ldots \xrightarrow{g_n} x_n\,.\]
We manipulate this chain, using the equivalence relation presented in Definition \ref{UFP}: Set
$h_{n-1}=g_n$, $h_{n-2}=g_ng_{n-1}$ etc., and $x'_{i}=h_ix_i$. Then we have the following commutative ladder diagram:
\begin{equation*}
\begin{tikzcd}[row sep=large]
x_0 \arrow[r, "g_1"] \arrow[d, "\mathrm{id}_{x_0}"] & x_1 \arrow[r, "g_2"] \arrow[d, "h_1"] & x_2 \arrow[d, "h_2"] \arrow[r, "g_3"] & \ldots  \arrow[r, "g_{n-2}"] & x_{n-2}\arrow[d, "h_{n-2}"] \arrow[r,"g_{n-1}"] & x_{n-1} \arrow[r, "g_n"] \arrow[d,"h_{n-1}"] & \phantom{\,.}x_n\phantom{\,.} \arrow[d,"\mathrm{id}_{x_n}"]\\
x_0 \arrow[r, "h_0"] & x'_1 \arrow[r, "1"] & x'_2 \arrow[r,"1"] & \ldots \arrow[r, "1"] & x'_{n-2} \arrow[r,"1"] & x'_{n-1} \arrow[r,"1"] & \phantom{\,.} x_n\,.
\end{tikzcd}
\end{equation*}

Thus, every chain  of morphisms is equivalent to one with morphisms labelled by $1$, except possibly the first morphism. With this observation, it is easily proved that a non-isomorphism $x \xrightarrow y$ is unfactorisable if and only if the chain $x<y$ is saturated, and  part (a) can be deduced.

For part (b), assume that $\P$ satisfies USC. Consider two chains of unfactorisables with common start and end point and the same composition. By the above argumentation, we can assume they have the following forms:
\[x_0 \xrightarrow{g} x_1 \xrightarrow{1} x_2 \ldots \xrightarrow{1} x_n\,,\]
and
\[x_0=x'_0 \xrightarrow{g'} x'_1 \xrightarrow{1} x'_2 \ldots \xrightarrow{1} x'_m=x_n\,.\]
Since the compositions are equal, we have $g=g'$. So we have the following chains in $\P$:
\[gx_0 < x_1 < x_2 < \ldots < x_n\]
and
\[gx_0 < x'_1 < x'_2 < \ldots < x'_n\,.\]
These are saturated since they correspond to unfactorisable morphisms. Since $\P$ satisfies USC, we have $m=n$ and $x_i=x'_i$, so the original chains of morphisms are equal.

Now, assume $\P\rtimes G$ has the UFP. Consider two saturated chains
\[x_0 < x_1 < x_2 < \ldots < x_n\]
and
\[x_0=x'_0 < x'_1 < x'_2 < \ldots < x'_m=x_n\,.\]
These yield the two horizontal rows of morphisms in the following diagram with common composition $1\in \mathrm{Hom}_{\P\rtimes G}(x_0, x_n)$, so we can invoke the UFP to get $n=m$ and the vertical isomorphisms making the diagram commute:
\begin{equation*}
\begin{tikzcd}[row sep=large]
x_0 \arrow[r, "1"] \arrow[d, "\mathrm{id}_{x_0}"] & x_1 \arrow[r, "1"] \arrow[d, "h_1"] & x_2 \arrow[d, "h_2"] \arrow[r, "1"] & \ldots  \arrow[r, "1"] & x_{n-2}\arrow[d, "h_{n-2}"] \arrow[r,"1"] & x_{n-1} \arrow[r, "1"] \arrow[d,"h_{n-1}"] & \phantom{\,.}x_n\phantom{\,.} \arrow[d,"\mathrm{id}_{x_n}"]\\
x_0 \arrow[r, "1"] & x'_1 \arrow[r, "1"] & x'_2 \arrow[r,"1"] & \ldots \arrow[r, "1"] & x_{n-2} \arrow[r,"1"] & x'_{n-1} \arrow[r,"1"] & \phantom{\,.} x_n\,.
\end{tikzcd}
\end{equation*}
By the commutativity of the diagram, one gets inductively that all $h_i=1$ and, since the corresponding arrows are isomorphisms, $x_i=x'_i$.
\end{proof}

An application of Theorem \ref{hereditary-minimal} yields the following characterisation:

\begin{thm} \label{transporter-hereditary} Let $k$ be a field, $G$ a group and $\P$ a $G$-poset satisfying \eqref{S} and ESC. Then $k(\P\rtimes G)$ is hereditary if and only if the following conditions hold:
\begin{itemize}
    \item $\P$ satisfies USC,
    \item for $x\in \P$, the stabiliser $\mathrm{Stab}_G(x)$ is countable locally $k^\times$-finite or the fundamental group of a connected graph of $k^\times$-finite groups,
    \item for $x<y$, $\mathrm{Stab}_G(x)\cap \mathrm{Stab}_G(y)$ is $k^\times$-finite.
\end{itemize}
\end{thm}

If we are interested in the global dimension of the category of right $k(\P\rtimes G)$-modules, we have to check the corresponding conditions for $(\P\rtimes G)^{\mathrm{op}}=\P^{\mathrm{op}}\rtimes G$. But the conditions are obviously the same for $\le$ and $\ge$, so we get:

\begin{cor}
$k(\P\rtimes G)$ is left hereditary if and only if it is right hereditary.
\end{cor}

\begin{cor} \label{transporter-subgroups-hereditary} Let $k$ be a field, $G$ a discrete group and $\F$ a family of finite subgroups, ordered by inclusion and equipped with the conjugation action of $G$. Then $k(\F\rtimes G)$ is hereditary if and only if
\begin{itemize}
\item $G$ is either countable locally $k^\times$-finite or the fundamental group of a connected graph of $k^\times$-finite groups,
\item all members of $\F$ are cyclic of prime power order, invertible in $k$, and their Weyl groups are $k^\times$-finite (except possibly for the Weyl group of $\{1\}$).
\end{itemize}
\end{cor}

\begin{proof} \emph{The 'only if' part.}  Suppose that $k(\F\rtimes G)$ satisfies the conditions from Theorem \ref{transporter-hereditary}. Let $K\in \F$. The USC, applied to $\{1\}$ and $K$, implies that $K$ is a $p$-group for some $p$  and can have at most one subgroup of each order. This implies by an easy induction that $K$ is cyclic of prime power order. $G$ is the stabiliser of $\{1\}$ and thus in one of Dicks' classes. The third condition of Theorem \ref{transporter-hereditary}, applied to the inclusion $\{1\}\subseteq K$ for $K\neq \{1\}$, implies that $N_G(K)$ is $k^\times$-finite. Thus the same holds for its subgroup $K$ and its quotient $W_G(K)$.

\emph{The 'if' part.} The USC follows easily from the fact that a cyclic group has at most one subgroup of any order, and the other two items of Theorem  \ref{transporter-hereditary} are trivially satisfied.
\end{proof}

\subsection{Orbit categories}
\label{orbit categories}

In this subsection we consider another example, orbit categories, which are widely applied while considering actions of groups on topological spaces, for instance, Bredon's coefficient systems and homology theory, approximation of classifying spaces of groups, etc. The reader with interest can refer to \cite{blo},  \cite{jmo} and \cite{davis-lueck}. One conceptual reason for the importance of the orbit category stems from Elmendorf's Theorem \cite{elmendorf} relating the homotopy theory of $G$-spaces to the homotopy theory of presheaves over the orbit category.

\begin{defi}
The \emph{orbit category} $\mathrm{Or}(G,\F)$ of $G$ with respect to the family $\F$ has as objects the transitive $G$-sets $G/K$, $K\in \F$, and as morphisms the $G$-equivariant maps between these.
\end{defi}

\begin{rema}
We recall from \cite[Lemma~6.4.1]{SW} that for $K, L\in \F$, there is an isomorphism
\begin{align*}
\phi^{L,K}\colon K\backslash \mathrm{Trans}_G(L,K) &\cong \mathrm{Hom}_{\mathrm{Or}(G,\F)}(G/L, G/K), \\
g &\mapsto \phi^{L,K}(g)
\end{align*}
with
\[\mathrm{Trans}_G(L,K)=\{g \in G; \, g L g^{-1}\subseteq K\}\]
and
\[(\phi^{L,K}(g))(xL)=xg^{-1}K\] for $x \in G$. These isomorphisms are compatible with composition in the obvious way. In particular, since $\F$ consists of finite groups only, there is an isomorphism of monoids
\[\mathrm{Hom}_{\mathrm{Or}(G,\F)}(G/K, G/K)\cong K\backslash N_G(K) = W_G(K)\,.\]
Thus, $\mathrm{Or}(G,\F)$ is an EI category. Also, since any noninvertible morphism strictly increases the cardinality of the finite isotropy group $H$, $\mathrm{Or}(G,\F)$ has the FFP.
\end{rema}

If we set $L=\{1\}$ in the above, we have $\mathrm{Trans}_G(\{1\},K)=G$ and get the isomorphism
\[\phi^{\{1\},K} \colon K\backslash G \cong \mathrm{Hom}_{\mathrm{Or}(G,\F)}(G/\{1\}, G/K)\,,\]
where $K\backslash G$ is furnished with the $(W_G(K),G)$-biset structure given by left and right multiplication. The biset stabiliser $H_1$ of $\phi^{\{1\},K}(1K)$ thus equals
\[H_1=\{([g], g^{-1}); g\in N_G(K)\}\]
which is isomorphic to $N_G(K)^{\mathrm{op}}$ via $\mathrm{pr}_2$, while $\mathrm{pr}_1$ is surjective onto $W_G(K)$. Consequently, we have the following lemma.

\begin{lemma}\label{orbit-A-B} Let $K \in \F$ with $K\neq \{1\}$.\\
(a) If Condition ($\mathrm{A}_d$) is satisfied for the $(W_G(K),G)$-biset $\mathrm{Hom}_{\mathrm{Or}(G,\F)}(G/\{1\},G/K)$, then $K$ is $k^\times$-finite.\\
(b) If Condition ($\mathrm{B}_d$) is satisfied for the $(W_G(K),G)$-biset $\mathrm{Hom}_{\mathrm{Or}(G,\F)}(G/\{1\},G/K)$, then $W_G(K)$ is $k^\times$-finite.\qedhere
\end{lemma}

Finally, we cite a result characterising when the orbit category has the UFP. It is proved in a similar way as Proposition~\ref{quillen-UFP} below and enables us to prove Theorem~\ref{orbit-hereditary} from the introduction.

\begin{citedprop}[{\cite[Prop.~6.5.5]{SW}}]
 \label{orbit-UFP} Let $G$ be a group and $\F$ a family of finite subgroups. Then $\mathrm{Or}(G,\F)$ has the UFP if and only if all $K\in \F$ are cyclic of prime power order, where different primes may occur.
\end{citedprop}

\begin{proof}[Proof of Theorem~\ref{orbit-hereditary}] If the two items are satisfied, all automorphism groups have hereditary group rings by Remark \ref{dicks-field} and $\mathrm{Or}(G,\F)$ has the UFP by Proposition~\ref{orbit-UFP}. Moreover, Conditions ($\mathrm{A}_d$) and ($\mathrm{B}_d$) are trivially satisfied: If $G/L < G/K$, and $L\neq \{1\}$, then $K$ is nontrivial, $H_i$ is a subgroup of the $k^\times$-finite group $W_G(K)\times W_G(L)$ and thus $k^\times$-finite. If $L=\{1\}$, then the biset stabilisers are all isomorphic to $N_G(K)$ by transitivity, and this is $k^\times$-finite since $K$ and $W_G(K)$ are.

Now, suppose that $k\mathrm{Or}(G,\F)$ is hereditary. Note that $G=W_G(1)$. By Remark~\ref{dicks-field}, if $k\mathrm{Or}(G,\F)$ is hereditary, then $G$ is either countable locally $k^\times$-finite or the fundamental group of a connected graph of $k^\times$-finite groups. By Lemma~\ref{orbit-A-B}~(b), all Weyl groups of nontrivial members of $\F$ are $k^\times$-finite. Finally, the members of $\F$ are cyclic of prime power order by Lemma \ref{orbit-UFP} and $k^\times$-finite by Lemma~\ref{orbit-A-B}~(a).
\end{proof}

\begin{exam} Suppose that $\F=\mathcal{FIN}$ is the family of all finite subgroups of $G$ and that $k\F$ is hereditary. If $G$ is locally $k^\times$-finite, then any two elements are contained in a finite, thus cyclic subgroup, and $G$ is abelian. Thus, $N_G(K)=G$ for an arbitrary $K$ and $G$ has to be $k^\times$-finite itself and thus cyclic.

On the other hand, if $G$ is the fundamental group of a connected graph of $k^\times$-finite groups, then one can prove that $G$ satisfies the above items if and only if this graph has trivial edge and loop groups. Contracting a spanning tree, we see that \emph{$G$ is a free product, finite or infinite, of finite groups $\mathbb{Z}/p_i^{k_i}$ where $p_i$ is a prime invertible in $k$, or $p_i=1$}.
\end{exam}

\begin{exam}
The groups $D_{\infty} \cong \mathbb{Z}/2 \ast \mathbb{Z}/2$ and $\mathrm{PSL}_2(\mathbb{Z})\cong \mathbb{Z}/2\ast \mathbb{Z}/3$ have a hereditary orbit category with respect to $\F=\mathcal{FIN}$.
\end{exam}

\begin{rema}
$G$ is the fundamental group of a \emph{finite} graph of finite groups if and only if it is virtually finitely generated free abelian \cite{kps}. In this case, we give a geometric characterisation of the Weyl group condition in Appendix~\ref{apx: trees}. It can be summarised as follows: If $F$ is a finite subgroup of $G$, then $W_G(F)$ is infinite if and only if $F$ fixes a ray (equivalently, a line) in the Bass-Serre tree, and there is a combinatorial algorithm how to read this off from the graph of groups.
\end{rema}

\begin{exam} The group $\SL_2(\mathbb{Z})\cong \mathbb{Z}/4\ast_{\mathbb{Z}/2}\mathbb{Z}/6$ has a nontrivial normal subgroup $\mathbb{Z}/2$ and thus the orbit category for any family containing this subgroup is not hereditary. However, the subgroup $\mathbb{Z}/3$ (canonically embedded via the second factor of the amalgam) has finite normaliser by Lemma~\ref{normalising-circle} and thus $\mathrm{Or}(\SL_2(\mathbb{Z}), \F_3)$ is hereditary where $\F_3$ denotes the family of subgroups which are finite $3$-groups.
\end{exam}

The discussion until here treated left $k\mathrm{Or}(G,\F)$-modules. Let us comment shortly on right $k\mathrm{Or}(G,\F)$-modules, i.\ e., left $k\mathrm{Or}(G,\F)^{\mathrm{op}}$-modules. Condition ($\mathrm{A}_d$), the UFP, and the Dicks condition for hereditarity of group rings are insensible when passing from a category to its opposite, but Condition ($\mathrm{B}_d$) a priori is not -- the two projections are interchanged. However, in the present case, the condition becomes stronger since the first projection ($\mathrm{pr}_2$ above) is an isomorphism in the case $L=\{1\}$. We thus get directly that $N_G(K)$ is $k^\times$-finite and consequently the Weyl group $W_G(K)$. We thus can prove the following in exactly the same way as Theorem~\ref{orbit-hereditary}:

\begin{cor} Let $k$ be a field, $G$ a discrete group and $\F$ a family of finite subgroups. Then $k\mathrm{Or}(G,\F)$ is right hereditary if and only if it is left hereditary, i.\ e., the conditions listed in Theorem~\ref{orbit-hereditary} hold.
\end{cor}

\subsection{Quillen categories}
\label{quillen categories}

We now discuss Quillen categories (also called Frobenius categories) as another important example where Theorem \ref{hereditary-minimal} can be applied. These categories are important for the notion of $p$-local finite groups, and were used by Quillen to prove the stratification theorem of group cohomology. For more details, the reader can refer to \cite{blo}.

\begin{defi}
The \emph{Quillen category} $\mathcal{Q}(G,\F)$ of $G$ with respect to $\F$ has the members of $\F$ as objects and \[\mathrm{Hom}_{\mathcal{Q}(G,\F)}(H,K)= \mathrm{Trans}_G(H,K)/C_G(H)\,.\]
\end{defi}

One sees in exactly the same way as for the orbit category that since $\F$ consists of finite groups, $\mathcal{Q}(G,\F)$ is an EI category with the FFP.

In contrast with the orbit category, the Quillen category has finite Hom sets, even if $G$ is infinite, since mapping $g$ to conjugation $c(g)$ by $g$ embeds $\mathrm{Hom}_{\mathcal{Q}(G,\F)}(H,K)$ into $\mathrm{Hom}_{\mathrm{Grp}}(H,K)$.

In the important special case that $G$ is a finite group, $P$ is a Sylow $p$-subgroup of $G$ and $\F$ is the family of subgroups of $P$ (and their conjugates in $G$), $\mathcal{Q}(G,\F)$ is called a fusion system. Fusion systems are of particular interest to group theorists and algebraic topologists, and have numerous applications in these areas. For a complete introduction, the reader can refer to \cite{ao}.

\begin{prop} \label{quillen-UFP} Let $G$ be a group and $\F$ a family of finite subgroups.
The category $\mathcal{Q}(G,\F)$ satisfies the UFP if and only if $\F$ consists only of cyclic subgroups of prime power order (where different prime bases may occur in the same family).
\end{prop}

\begin{proof}
\emph{The 'only if' part.} Suppose that $\mathcal{Q}(G,\F)$ has the UFP.  First of all, it can be proved easily that every member $K$ of $\F$ is a finite $p$-group (for some $p$). We now show that $K$ has only one maximal subgroup. Indeed, let $H'$ and $H''$ be two maximal subgroups of $K$. Extend these to chains
\[\{1\}\subseteq H'_1 \subseteq H'_2 \subseteq \ldots \subseteq H'_{n-1} = H'\subseteq H'_n=K\]
and
\[\{1\}\subseteq H''_1 \subseteq H''_2 \subseteq \ldots \subseteq H''_{m-1} = H''\subseteq H''_m=K\]
of subgroups in $\F$, where $H'_i\subseteq H'_{i+1}$ and $H''_i\subseteq H''_{i+1}$ are maximal. By the UFP, we have $m=n$ and a commutative ladder
\begin{equation*}
\begin{tikzcd}[row sep=large]
\{1\} \arrow[r, "1"] \arrow[d, "1"] & H'_1 \arrow[d, "h_1"] \arrow[r,"1"] & H'_2 \arrow[d, "h_2"] \arrow[r,"1"]& \ldots \arrow[r,"1"] & H'_{n-2} \arrow[d, "h_{n-2}"] \arrow[r,"1"] & H' \arrow[d, "h"] \arrow[r,"1"] & K \arrow[d,"1"]\\
\{1\} \arrow[r,"1"] & H''_1 \arrow[r,"1"] & H''_2 \arrow[r,"1"]& \ldots \arrow[r,"1"]& H''_{n-2} \arrow[r,"1"] & H''\arrow[r,"1"] & K
\end{tikzcd}
\end{equation*}
From the commutativity of the rightmost square, we have $h\in C_G(H')$ and thus, since $h\in \mathrm{Trans}_G(H',H'')$, we get $H'=H''$.

This shows that $K$ is cyclic by an easy induction: The center $C$ of $K$ is non-trivial since $K$ is a $p$-group. The quotient $K/C$ still only has one maximal subgroup, thus is cyclic. It is a well-known fact that cyclicity of $K/C$ forces $K$ to be abelian, i.\ e. a product of cyclic groups. Again since $K$ has only one maximal subgroup, it is itself cyclic.

\emph{The 'if' part.} Now, suppose that $\F$ only has cyclic members of prime power order. Given a chain
\[H_0 \xrightarrow{g_1} H_1 \xrightarrow{g_2} H_2  \ldots \xrightarrow{g_n} H_n\]
of unfactorisable morphisms, we first manipulate it to get a chain
\[H_0 \xrightarrow{g'} H'_1 \xrightarrow{1} H'_2 \xrightarrow{1} H'_3 \ldots H'_{n-1} \xrightarrow{1} H_n\] with composition $g'$ modulo $C_G(H_{0})$. This works similarly as in the beginning of the proof of Lemma \ref{ESC,USC}. Let \[H_0 \xrightarrow{g''} H''_1 \xrightarrow{1} H''_2 \xrightarrow{1} H''_3 \ldots H''_{n-1} \xrightarrow{1} H_m=H_n\]
be another chain with the same composition, i.\ e. $g'(g'')^{-1} \in C_G(H_0)$.

First of all, let us note that since all inclusions are unfactorisable and all $H'_i$ and $H''_i$ are cyclic of order a power of $p$, the indices $[H'_i:H'_{i-1}]$ as well as $[H''_i:H''_{i-1}]$ are all equal to $p$, so that $m=n=\mathrm{log}_p([H_n:H_0])$. Since $H_n$ is cyclic, it has at most one subgroup of every order, so $H'_i=H''_i$. It follows that the two ladders equal each other since $g'= g'' \in C_G(H_0) \backslash \mathrm{Trans}_G(H_0,H'_1)$.
\end{proof}

\begin{thm}\label{quillen-hereditary}
Let $k$ be a field, $G$ a discrete group and $\F$ a family of finite subgroups. Then $k\mathcal{Q}(G,\F)$ is left hereditary if and only if all $K\in \F$ are cyclic of prime power order, and the finite groups $N_G(K)/C_G(K)$ have order invertible in $k$. Moreover, $k\mathcal{Q}(G,\F)$ is right hereditary if and only if it is left hereditary.
\end{thm}

\begin{proof}
Since all automorphism groups $N_G(K)/C_G(K)$ are finite, Dicks' condition from Remark \ref{dicks-field} reduces to them being $k^\times$-finite. Under this assumption, Conditions $(\mathrm{A}_d)$ and $(\mathrm{B}_d)$ become vacuous.
\end{proof}

\begin{rema}
A slight variation of the Quillen category is the subgroup category in the sense of L\"uck \cite{lueck-cc}. Its morphism sets are given by
 \[\mathrm{Hom}_{\mathrm{Sub}(G,\F)}(H,K)= K\backslash \mathrm{Trans}_G(H,K)/C_G(H)\,.\]
One can prove that $k\mathrm{Sub}(G,\F)$ is left hereditary if and only if it is right hereditary if and only if the conditions of Proposition~\ref{quillen-hereditary} hold. One sees again that dividing out the finite group $K$ doesn't make a difference, compare also Corollary~\ref{transporter-subgroups-hereditary} and Theorem~\ref{orbit-hereditary}.
\end{rema}

\section{An alternative proof of sufficiency and structure theorems for projective modules}
\label{suff2}

This section contains a direct combinatorial proof of Theorem \ref{bound gldim}, with the addendum of Theorem \ref{intro-sum-of-rep}, but under certain mild combinatorial restrictions on $\C$.
We assume throughout that $\C$ is the free tensor category associated to some tensor quiver $(X,\mathcal{U})$, that $\gldim R_x\le N$ for all $x\in X$, and that conditions $(\mathrm{A}_N)$, $(\mathrm{B}_N)$ and (BP) introduced in Subsection~\ref{ANBN} are satisfied. The additional combinatorial conditions will be collected along the way. They can be formulated most generally in terms of the quiver $(X,\mathcal{U})$ \emph{and} the module we are considering, cf. Corollaries \ref{surjective}, \ref{semi-N-hereditary} and \ref{semi-hereditary}, where the latter two deal with finitely generated modules only. Finally, we show that our reasoning can be applied under the simple condition that $X$ has no left-infinite chains. The proof of this assertion is non-trivial (to our knowledge) and uses the so-called combinatorial compactness argument.

Recall that for a $k\C$-module $K$, and for $d$ an object of $\C$, the left $R_d$-submodule of $K(d)$ generated by the images of $K(c)$ for $c<d$ is denoted by $B_d(K)$. This defines a subfunctor of $K$ mapping $d$ to $B_d(K)$. The left $R_d$-module $S_d(K)$ is defined by the short exact sequence
\begin{align}\label{BKS}
    0\longrightarrow B_d(K) \longrightarrow K(d) \longrightarrow S_d(K)\longrightarrow 0\,.
\end{align}

\begin{lemma}\label{split}
Let $(X,\mathcal{U})$ satisfy $(\mathrm{B}_N)$ and (BP). Assume $N\ge 1$ and let $K$ be an $N$-th kernel. Then the exact sequence \eqref{BKS} splits for every object of $\C$.
\end{lemma}

\begin{proof}
Since $N\ge 1$, $K$ is a submodule of the projective $\C$-module $P_N$. By adding a projective module to $P_N$ and $P_{N-1}$ (or $M$ if $N=1$), we may assume that $P_N$ is a free module, i.\ e.
\begin{align}\label{free}P_N(y) = \bigoplus_{x<y} \C(x,y)^{(A)} = \bigoplus_{\gamma\colon ? \rightarrow y}  U(\gamma)^{(A)}  \end{align}
where $A$ is some (possibly infinite) set. The subscript of the latter sum indicates that it runs over all chains ending in $y$. Dividing these up according to the second-to-last entry (which exists in every case except for the singleton chain $(y)$, and the empty chain which doesn't contribute to the sum), we can write
\[P_N(y)\cong R_y^{(A)} \oplus \bigoplus_{x<y} U(x,y)\otimes_{R_x} \left(\bigoplus_{\gamma\colon ?\rightarrow x} U(\gamma)^{(A)}\right) = R_y^{(A)}\oplus\,\bigoplus_{x<y} U(x,y)\otimes_{R_x} P_N(x)\,.\]
Now, consider for $x<y$ the short exact sequence
\[0 \longrightarrow K(x) \longrightarrow P_N(x) \longrightarrow P_N(x)/K(x)\longrightarrow 0\,.\]
By Lemma \ref{objectwise-projective}, the cokernel is an $(N-1)$-st kernel in the category of $R_x$-modules and $K(x)$ is an $N$-th kernel, thus projective. Applying $(\mathrm{B}_N)$, the sequence splits after tensoring with $U(x,y)$:
\[U(x,y)\otimes_{R_x} P_N(x) \cong U(x,y)\otimes_{R_x} \left(P_N(x)/K(x)\right) \oplus G_x(U(x,y)\otimes_{R_x} K(x))\]
where
\[G_x\colon U(x,y)\otimes_{R_x} K(x) \longrightarrow U(x,y)\otimes_{R_x} P_N(x)\]
is the canonical map which is not necessarily injective. It follows that
\begin{align*} P_N(y)&\cong R_y^{(A)} \oplus \bigoplus_{x<y} U(x,y)\otimes_{R_x} \left(P_N(x)/K(x)\right) \oplus  \bigoplus_{x<y} G_x(U(x,y)\otimes_{R_x} K(x))\notag\\ & \cong   R_y^{(A)} \oplus \bigoplus_{x<y} U(x,y)\otimes_{R_x} \left(P_N(x)/K(x)\right) \oplus B_y(K) \end{align*}
as left $R_y$-modules. Now, $K(y)$ is an $R_y$-submodule of $P_N(y)$ \emph{which contains the third summand $B_y(K)$ completely} and thus splits of this summand.
\end{proof}

In the situation of Lemma \ref{split}, we thus get a map of $k\C$-modules
\begin{align}\label{def-F}
    F\colon \bigoplus_{x} k\C\otimes_{R_x} S_x(K) \rightarrow K\,.
\end{align}

\begin{rema} This map is not natural in $K$ since it depends on choices in splitting the sequences \eqref{BKS}.
\end{rema}

\begin{lemma}\label{F-injective}
Suppose that $(\mathrm{A}_N)$ holds. Then $F$ is injective.
\end{lemma}

\begin{proof}
The map $G_x$ in the proof of Lemma \ref{split} is injective, so
\[B_y(K)\cong \bigoplus_{x<y} U(x,y)\otimes_{R_x} K(x)\]
and we have
\[K(y)\cong S_y(K)\oplus \bigoplus_{x<y} U(x,y)\otimes_{R_x} K(x)\,.\]
Iterating this, we get
\[K(y)\cong \bigoplus_{\substack{\gamma\colon x_0\rightarrow y\\\ell(\gamma)\le \ell}} U(\gamma)\otimes_{R_{x_0}} S_{x_0}(K) \oplus \bigoplus_{\substack{\gamma \colon x_0\rightarrow y\\ \ell(\gamma)=\ell+1}} U(\gamma)\otimes_{R_{x_0}} K(x_0)\,. \]
The short exact sequences
\[0\longrightarrow \bigoplus_{\substack{\gamma\colon x_0\rightarrow ?\\\ell(\gamma)\le \ell}} U(\gamma)\otimes_{R_{x_0}} S_{x_0}(K)\longrightarrow K \longrightarrow \bigoplus_{\substack{\gamma \colon x_0\rightarrow ?\\ \ell(\gamma)=\ell+1}} U(\gamma)\otimes_{R_{x_0}} K(x_0) \longrightarrow 0 \]
of $k\C$-modules (which are not natural in $K$ and split objectwise, but not as functors) are connected by maps from the $\ell$-th to the $(\ell+1)$-st sequence given by the inclusion on the left, which is the restriction of $F$, the identity in the center and the induced map on the right. Passing to the colimit, since filtered colimits are exact, we get an exact sequence
\begin{align}\label{magic-ses} 0\longrightarrow \bigoplus_{\substack{\gamma\colon x_0\rightarrow ?}} U(\gamma)\otimes_{R_{x_0}} S_{x_0}(K)\longrightarrow K \longrightarrow \colim_{\ell} \bigoplus_{\substack{\gamma \colon x_0\rightarrow ?\\ \ell(\gamma)=\ell+1}} U(\gamma)\otimes_{R_{x_0}} K(x_0) \longrightarrow 0\,.\end{align}
The left map is identified with $F$, so that $F$ is injective.
\end{proof}

Let us consider the colimit on the right in the last short exact sequence of the proof in more detail. The structure maps
\begin{align}\label{colim-map}
\psi_\ell\colon \bigoplus_{\substack{\gamma \colon x_0\rightarrow ?\\ \ell(\gamma)=\ell+1}} U(\gamma)\otimes_{R_{x_0}} K(x_0) \longrightarrow \bigoplus_{\substack{\gamma \colon z_0\rightarrow ?\\ \ell(\gamma)=\ell+2}} U(\gamma)\otimes_{R_{z_0}} K(z_0)
\end{align}
come from the (chosen) splittings
\[K(x_0)\cong S_{x_0}(K)\oplus \bigoplus_{z_0<x_0} U(z_0,x_0)\otimes_{R_{z_0}} K(z_0) \]
by projecting to the second summand and thus have the following property:
\begin{equation}\tag{$\ast$}\label{ast}
\parbox{\dimexpr\linewidth-4em}{A summand on the left-hand side of \eqref{colim-map} indexed by a chain starting in $x_0$ is mapped only to summands indexed by chains starting in some $z_0 < x_0$.}
\end{equation}

This property directly implies the following:

\begin{cor} \label{surjective}
Suppose that for every $y$, there is a number $M$ such that for all chains $\gamma=(x_0,\ldots, y)$ of length $\ge M$, we have $U(\gamma)\otimes_{R_{x_0}} K(x_0)=0$. Then $F$ is surjective.
\end{cor}

Note that if we can prove that $F$ is bijective, this doesn't only imply that the $N$-th kernel $K$ is projective, but also that it is of a very specific form, namely of the form
\begin{align} \label{sum-of-representables} K\cong \bigoplus_{x} k\C\otimes_{R_x} P_x \end{align}
for some projective $R_x$-modules $K$. Let us call a $\C$-module of this form an \emph{induced module}. In particular, if we prove surjectivity of $F$ for all $N$-th kernels $K$, we may take $K$ to be an arbitrary projective and get that every projective is an induced module.

\begin{rema}\label{combinatorial-conditions}
Here are two cases where the condition of Corollary \ref{surjective} is satisfied:
\begin{itemize}
    \item For every $y\in X$, the length of chains (with respect to $<$) terminating in $y$ is bounded.
    \item $K$ is finitely generated and for all $x, y\in X$, the length of chains between $x$ and $y$ is bounded.
\end{itemize}
\end{rema}

We state a corollary in the second case, but leave out the first case since we will prove a more general version below in Proposition~\ref{F-surjective-tychonoff}.

\begin{cor} \label{semi-N-hereditary}
Let $N\ge 1$. Let $\C$ be the free tensor category over a directed tensor quiver satisfying the \emph{second} condition of Remark \ref{combinatorial-conditions} (''locally bounded``), $(\mathrm{A}_N)$, $(\mathrm{B}_N)$, and $\gldim R_x\le N$ for all $x$. Then the category of left $k\C$-modules has the property that any finitely-generated $N$-th kernel is projective. Furthermore, any finitely generated projective module is an induced module as in \eqref{sum-of-representables}.
\end{cor}

For $N=1$, note that a $1$-st kernel is just a submodule of a projective module. The property that finitely generated submodules of projective modules are projective is sometimes called semi-hereditarity.

\begin{cor} \label{semi-hereditary}
Let $\C$ be the free tensor category over a directed tensor quiver satisfying the \emph{second} condition of Remark \ref{combinatorial-conditions} (''locally bounded``), (A), (B), and $\gldim R_x\le 1$ for all $x$. Then the category of left $k\C$-modules is semi-hereditary, and any finitely generated projective module is an induced module as in \eqref{sum-of-representables}.
\end{cor}

We will now give a proof that $F$ is bijective which imposes less severe combinatorial conditions on $\C$, but uses more advanced combinatorial arguments. Actually, we only assume the following:

\begin{defi}\label{left-infinite chains}
Let $X$ be a partially ordered set. A \emph{left-infinite chain in $X$} is a sequence $(x_i)_{i\ge 0}$ in $X$ indexed by the non-negative integers such that $x_{i}<x_{i-1}$.
\end{defi}

\begin{prop}\label{F-surjective-tychonoff}
Suppose that $X$ has no left-infinite chains. Then $F$ is surjective (i.\ e., bijective).
\end{prop}

\begin{proof}
We want to show that \[\colim_\ell \bigoplus_{\substack{\gamma \colon x_0\rightarrow ?\\ \ell(\gamma)=\ell+1}} U(\gamma)\otimes_{R_{x_0}} K(x_0)=0\]
where the product is along the structure maps $\psi_{\ell}$ described in \eqref{colim-map} above. The only thing we will use about the $\psi_{\ell}$ is property ($\ast$) displayed above.

Let $x\in \bigoplus_{\substack{\gamma \colon x_0\rightarrow ?\\ \ell(\gamma)=\ell+1}} U(\gamma)\otimes_{R_{x_0}} K(x_0)$ for some $\ell$. We want to show that for some $L$, the image of $x$ vanishes in $\bigoplus_{\substack{\gamma \colon x_0\rightarrow ?\\ \ell(\gamma)=L+1}} U(\gamma)\otimes_{R_{x_0}} K(x_0)$. For all $n\ge \ell$, denote by $X_n$ the set of objects $z$ of $\C$ such that $\psi_n(\psi_{n-1}(\ldots(\psi_{\ell}(x))))$ -- the image of $x$ under iteration of $\psi$ -- has a nonzero component indexed by a chain starting in $z$.

We know that $X_n$ is finite (since the domain is a direct sum). By property ($\ast$), we know that for every $z\in X_n$, there is a $z'\in X_{n-1}$ with $z<z'$. Iterating this, we get for every $z\in X_n$ a chain $z=z_n<z_{n-1} <\ldots < z_\ell$ with $z_i\in X_i$. The lemma below finishes the proof. \end{proof}

\begin{lemma}\label{tychonoff}
Let $X$ be as in the proof of Proposition \ref{F-surjective-tychonoff}. We have $X_L=\emptyset$ for $L$ large.
\end{lemma}

We separated this lemma since its proof is purely combinatorially. Actually, it uses a so-called compactness argument building on Tychonoff's Theorem that the product of compact topological spaces is compact. This is a classical combinatorial argument which can be traced back to \cite{koenig}. The first author thanks Jens Reinhold for teaching him how to apply compactness arguments many years ago.

\begin{proof}
Give every set $X_i$ the discrete topology and the product $K=\Pi_{i=\ell}^{\infty} X_i$ the product topology. It is compact by Tychonoff's Theorem since all $X_i$ are finite. Set
\[K_j=\{(z_i)\in K\mid z_i < z_{i-1} \,\mbox{for}\, \ell+1 \le i\le j\}\,.\]
The $K_i$ are closed in $K$. Also,
\[\bigcap_{i=\ell}^{\infty} K_i=\emptyset \]
since $\C$ has no left-infinite chains. Since $K$ is compact, this implies $K_L=\emptyset$ for some $L$. But we argued above that an element of $X_n$ always gives a chain in $K_n$. Thus, $X_L=\emptyset$.
\end{proof}

The following corollary includes Theorem \ref{intro-sum-of-rep} from the introduction.

\begin{cor} \label{equal_gl_dim-tychonoff} Let $N\ge 1$.
Let $\C$ be the free tensor category over a directed tensor quiver without left-infinite chains, satisfying $(\mathrm{A}_N)$, $(\mathrm{B}_N)$ and (BP), and $\gldim R_x\le N$ for all $x$. Then the global dimension of the category of left $k\C$-modules is at most $N$, and every projective module is an induced module as in \eqref{sum-of-representables}.
\end{cor}

\section{Categories without the FFP}
\label{without-FFP}

In this final section, $\C$ is a discrete EI category which does not necessarily have the finite factorisation property FFP. We deal with the question
how to characterise hereditarity of $k\C$ in this case. We find a generalisation of the UFP in terms of a certain poset $\Theta(\alpha)$ which we can prove to be necessary, cf. Lemma~\ref{Theta-well-ordered} and Proposition~\ref{Theta-totally-ordered}. In all examples the authors have constructed where these combinatorial conditions, together with the usual algebraic conditions that the $kG_c$ are hereditary, and ($\mathrm{A}_d$) and $(\mathrm{B}_d)$, are satisfied, we could prove hereditarity of $k\C$, using the methods of Section~\ref{suff2}. However, we could not succeed to derive a general proof. In particular, it is not clear whether an additional algebraic condition is necessary in the non-FFP case, as discussed in Remark~\ref{splitting-off-infinite-direct-sum}.

We begin with two examples of categories without the FFP.

\begin{exam}
Consider the poset $\C$ of real numbers between $0$ and $1$ with the usual ordering. Viewed as an EI category, $\C$ has no unfactorisable morphisms, and its category algebra $k\C$ (with $k$ a field, say) is not hereditary, as follows from Proposition~\ref{Theta-well-ordered} below.

Instead of all real numbers in $[0,1]$, we could have worked with \nobreak{$\{0\}\cup \{\frac 1n, n\in \mathbb{N}\}$} here. In this case, there are some unfactorisable morphisms, but all non-invertible morphisms with source $0$ cannot be factored into a finite product of unfactorisables.
\end{exam}

\begin{exam}\label{hereditary-no-FFP}
Consider the partially ordered set $\{1-\frac{1}{n}\}\cup \{1\}$, or equivalently $\mathbb{N}\cup \{\infty\}$, and let $\C$ again denote the associated category which still doesn't have the FFP. If $k$ denotes a semisimple commutative ring, then $k\C$ is hereditary. This can be proved with the methods of Section \ref{suff2}.
\end{exam}

\begin{rema}
The methods of our first proof in Section \ref{sufficiency} are not available here since  Proposition~\ref{omega-of-tensor} doesn't apply: $k\C$ is not a tensor algebra any more. It would be interesting to find a description $\Omega_R^1 k\C$ in this situation.
\end{rema}

We will now derive a combinatorial necessary condition which presumably plays the role of the UFP in the non-FFP case, but is more difficult to formulate.

\begin{defi} \label{defi:Theta} For any morphism $\alpha\colon x\rightarrow y$ in $\C$,
let $\Theta(\alpha)$ denote the following poset. Elements are equivalence classes of triples $(t,g,f)$, where $t$ is an object of $\C$, $f\in \C(x,t)$ and $g\in \C(t,y)$, satisfying $\alpha=g\circ f$. The equivalence relation is as follows: $(t,g,f)\sim (t',g',f')$ if there is an isomorphism $h\colon t\cong t'$ such that $g'=g\circ h^{-1}$ and $f'=h\circ f$.

We order $\Theta(\alpha)$ by setting $[t,g,f]\le [t',g',f']$ if there exists $h\colon t\rightarrow t'$ such that $f'=h\circ f$ and $g=g'\circ h$.
\end{defi}

It is easily checked that the ordering on $\Theta(\alpha)$ is a well-defined partial ordering.

\begin{exam}
Suppose that $\alpha=u_n\circ u_{n-1}\circ \ldots \circ u_1$ is a finite product of unfactorisables. Then $\Theta(\alpha)$ contains the finite chain
\[[x,\alpha, \mathrm{id}]< [x_1, u_n\circ u_{n-1}\circ \ldots \circ u_2, u_1] < \ldots < [x_{n-1}, u_n, u_{n-1}\circ \ldots \circ u_1] < [y,\mathrm{id},\alpha]\]
where $u_i\colon x_{i-1}\rightarrow x_i$, $x_0=x$ and $x_n=y$.

If the above factorisation is essentially unique in the sense of Def.~\ref{UFP} -- for instance if $\C$ has the UFP -- then $\Theta(\alpha)$ consists of this chain only. Conversely, one can easily prove that if $\Theta(\alpha)$ consists only of this chain, then the factorisation is essentially unique.

Consequently, if $\C$ has the FFP, then $\C$ has the UFP if and only if each $\Theta(\alpha)$ consists of a single finite chain.

In absence of the FFP, this is substituted by the conditions of being totally ordered and well-ordered, as the following two results show.
\end{exam}

\begin{lemma}\label{Theta-totally-ordered}
If $k\C$ is hereditary, then $\Theta(\alpha)$ is totally ordered for all $\alpha$. In particular, $[t,g,f]\le [t',g', f']$ if and only if $t\le t'$.
\end{lemma}

\begin{proof}
This works exactly as the first step of the proof of Proposition \ref{UFP-uniqueness} (before the induction on the length of $\alpha$).
\end{proof}

\begin{prop}\label{Theta-well-ordered}
If $k\C$ is hereditary, then $\Theta(\alpha)$ is well-ordered for all $\alpha$.
\end{prop}

\begin{proof}
Let $S$ be a subset without minimal element. Let $K$ denote the following subfunctor of $\C(x,-)$:
\[K(z) = k\left  \{ g\colon x\rightarrow z;\, \exists\, [t,g,f]\in S,\, h\colon t\rightarrow z \,\,\mathrm{s. t.}\,\, g=h\circ f \right \}\,.\]
We show that $K$ is not projective, so that $k\C$ will not be hereditary.

Let $B=\bigoplus\limits_{[t,g,f]\in S} \C(t,-)$. This has an obvious surjection $\pi$ to $K$ given by precomposition with $f$. We claim that $\mathrm{Hom}_\C(K,B)=0$. This implies that $\pi$ doesn't have a section and $K$ is not projective.

Consider a natural transformation $F\colon K\rightarrow B$. Choose an arbitrary $[t,g,f]\in S$ and consider $f\in K(t)$. Then $F(f)\in B(t)$ has only finitely many nonzero components. Let $[t_2,g_2, f_2]\in S$ denote the minimal index of a nonzero component. Since $S$ has no minimal element, there is $[t_1,g_1,f_1]$ which is strictly smaller than $[t_2,g_2, f_2]$ and $[t,g,f]$. Let $h\colon t_1\rightarrow t$ with $f=h\circ f_1$.

Consider the following commutative diagram:

\begin{equation*}
\begin{tikzcd}[row sep=large]
K(t_1) \arrow[r, "K(h)"] \arrow[d, "F"] & K(t) \arrow[d,"F"]\\
B(t_1) \arrow[r,"B(h)"] & B(t)
\end{tikzcd}
\end{equation*}
and start with $f_2$ in the upper right corner. Going via $K(t)$, it is first mapped to $f$ and then to $F(f)$. Going via $B(t_2)$, by definition of $B$, it is mapped to an element of
\[B(t)= \bigoplus\limits_{[t',g',f']\in S} \C(t',t)\]
which is concentrated in components indexed by elements $t'\le t_1 <t_2$. On the other hand, $F(f)$ only has nonzero components indexed by elements $t' \ge t_2$. Thus, $F(f)=0$. By definition of $K$, it follows that $F=0$.
\end{proof}

\begin{rema}
The last two results imply that we can associate to every morphism $\alpha$ a unique 'infinite factorisation into unfactorisables'. Indeed, we can extract for any element $\kappa+1\in \Theta(\alpha)$ which is not a limit element an unfactorisable morphism $u_{\kappa+1}\colon x_{\kappa}\rightarrow x_{\kappa+1}$  such that $f_{\kappa+1}=u_{\kappa+1}\circ f_{\kappa}$. This gives us a sequence $(u_{\kappa})_{\kappa}$ which is unique up to the equivalence relation known from the FFP case (cf.\ Def.~\ref{UFP}), applied at any intermediate non-limit element, and we may think of $\alpha$ as the 'infinite composition of the $u_\kappa$'.

However, the existence of limit elements in $\Theta(\alpha)$ makes this point of view problematic. In particular, there are the following three important caveats:
\begin{itemize}
    \item A sequence of unfactorisables might or might not have an 'infinite composition'.
    \item There might be two different morphisms $\alpha, \beta\colon x\rightarrow y$ with the same $u_\kappa$ sequences. This cannot happen in the FFP case since $\alpha$ is always the composition of the finitely many $u_\kappa$.
    \item The stabliser of $\alpha$ under the right $G_y$-action is not determined by the $u_\kappa$-sequence. This is also in contrast to the situation where $\C$ satisfies the FFP and UFP.
\end{itemize}
We want to stress that it is possible to construct examples demonstrating these phenomena such that the category in question has hereditary category algebra, so that these are important for the sake of this paper.
\end{rema}

\begin{rema} \label{splitting-off-infinite-direct-sum}
Finally, we want to point to a possible additional algebraic obstruction occurring in the case of categories without the FFP. Let us consider a directed $k$-linear category $\C$ which has $\mathbb{N}\cup \{\infty\}$ as underlying poset of objects as in Example~\ref{hereditary-no-FFP}, but now with arbitrary hereditary endomorphism rings $R_n$, $n\in \mathbb{N}$, and $R_{\infty}$ at the objects. Suppose we have a submodule $K$ of the projective module $\C(1,-)$ and want to prove that it is projective. Going along the same route as in the proof of Lemma \ref{split}, we can show that \[K(n) \cong B_n(K) \oplus S_n(K)\]
for every $n$. Let us now consider the situation at the object $\infty$. Denoting by $S_n$ the image of $S_n(K)$ in $K(\infty)$, we can show (if Conditions (A) and (B) are satisfied) that $K(\infty)$ splits off its submodule $S_1\oplus \ldots \oplus S_n$ for every $n$. However, for projectivity of $K$, it is necessary and sufficient that $K$ splits off $\bigoplus_{n=1}^{\infty} S_n$. In general, if a module $X$ splits off submodules $S_1\oplus \ldots \oplus S_n$, this does \emph{not} imply that it splits off the infinite direct sum $\bigoplus_{n=1}^{\infty} S_n$ (see \cite{mo} for an example with $R$ hereditary and $X$ projective), but it is not clear to the authors whether it is possible to construct an example where this effect impedes hereditarity of $\C$. It thus remains unclear whether an additional algebraic condition is necessary at this point.
\end{rema}

\begin{appendix}
\section{Normalisers in groups acting on trees}
\label{apx: trees}

Let $\pi=\pi_1(G,Y,P_0)$ be the fundamental group of a connected \emph{finite} graph of $k^\times$-finite groups $(G,Y)$. In this appendix, we analyse the condition appearing in Theorem~\ref{orbit-hereditary} that a finite subgroup  $F\subseteq \pi$ has finite normaliser $N_\pi(F)$. We first treat this question combinatorially in terms of the graph of groups, and then geometrically. The geometry enters by the well-known fact that $\pi$ acts on the Bass-Serre tree $X=\widetilde{X}(G,Y,T)$ with finite stabilisers and finite quotient. Here $T$ is a spanning tree of the quotient graph $Y$. We use the notation, constructions and main results of the famous and beautiful book \cite{trees}.

We now explain how to read off the normaliser of a finite subgroup  $F\subseteq \pi$ from the graph of groups. By conjugating (in $\pi$) if necessary, we may assume that $F$ fixes a vertex in the chosen lift of $T$ to $X$.

Let $c$ be a path in $Y$, given by edges $y_1, \ldots, y_n$. We put $P_i=t(y_i)=o(y_{i+1})$.  Recall that a \emph{word of type $c$} is a pair $(c,\mu)$ where $\mu=(r_0,\ldots, r_n)$ with $r_i\in G_{P_i}$. It is reduced if $n=0$ and $r_0\neq 1$, or if $n>0$ and whenever $y_{i+1}=\overline{y_i}$, we have $r_i\notin G_{y_i}^{y_i}$. Every reduced word with $c$ a circle is nontrivial in $\pi$, and every word in $\pi$ can be written as a reduced word. If we restrict to paths starting and ending in $P_0$, as we do in $\pi_1(G,Y,T)$, then $c$ is unique and $\mu$ is unique up to the equivalence relation generated by
\begin{align}\label{equivalence-reduced-words}(r_0, \ldots, r_n) \sim (r_0, \ldots, r_i a^{\overline{y_{i+1}}}, (a^{y_{i+1}})^{-1}r_{i+1}, \ldots, r_n)\end{align}
with $a\in G_{y_{i+1}}$ \cite[p.~50]{trees}.

Let $(c,\mu)$ be a reduced word as above. For $0\le i \le n$, let $_{i}c$ denote the starting segment $(y_1, \ldots y_i)$ of $c$, and $_i\mu$ the starting segment $(r_0,\ldots r_i)$ of $\mu$. We view $(_i c, _i\mu)$ as a reduced word centered at $P_0$ by going the same path backwards with trivial labels.

\begin{lemma} \label{normalising-circle}
Let $F\subseteq G_{P_0}$ be a finite subgroup, and let $(c,\mu)$ be a reduced word such that $|c,\mu|$ normalises $F$. Then
\[|_i c,_i\mu|^{-1} F |_i c,_i\mu| \subseteq G_{\overline{y_{i+1}}}^{\overline{y_{i+1}}}\,. \]
\end{lemma}

Here $G_{\overline{y_{i+1}}}^{\overline{y_{i+1}}}$
denotes the image of $G_{\overline{y_{i+1}}} = G_{y_{i+1}}$ in $G_{t(\overline{y_{i+1}})} = G_{o(y_{i+1})}=G_{P_i}$ as usual. Intuitively, the lemma says that every element normalising a subgroup $F$ describes a way how to move this subgroup along the graph of groups, starting in $P_0$ and inserting conjugations at subsequent vertices if necessary to push it into the next edge group. In the end, we arrive at $P_0$ again, with a subgroup conjugate to $F$ in $G_{P_0}$.

\begin{proof}
Let $g=|c,\mu|$. Then $F$ fixes $P_0$ and $gP_0$, thus the geodesic between the two in $X$. We show how to use this geodesic to write $g$ in reduced form satisfying the conjugation assertion of the lemma. By uniqueness of reduced word presentations, up to equivalence as described above, this proves the lemma.

Let $(z_1, \ldots, z_m)$ be the image of the geodesic from $P_0$ to $gP_0$ in $Y$, with $Q_i=t(z_i)=o(z_{i+1})$. We have $Q_0=Q_m=P_0$. We can write the first edge of the path in $X$ as $s_0z_1$ with $s_0\in G_{P_0}$ since it is incident to $P_0$. The second vertex then equals $s_0Q_1$. The second edge can be written as $(s_0z_1s_1) z_2$ with $s_1\in G_{Q_1}$ since it is incident to $s_0Q_1$. Here, $s_0z_1s_1$ is to be understood as an element of $\pi$ (we supress inserting the path back to $P_0$ with trivial labels in this proof) which acts on the edge $z_2$. Inductively, one gets a description of the path as in Figure~\ref{path}.

\begin{figure}
    \centering
    \includegraphics[width=\textwidth]{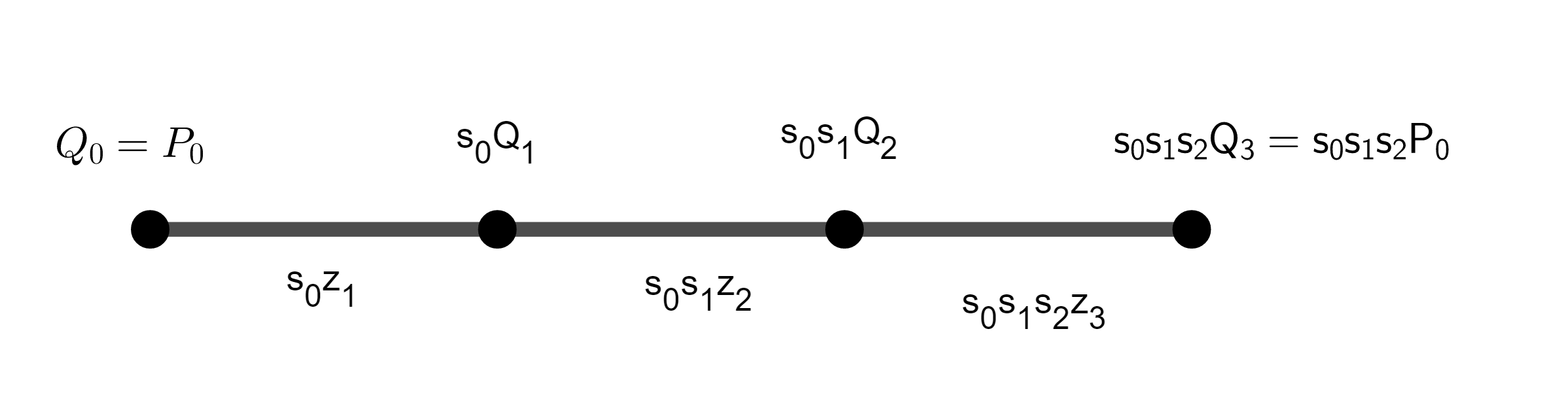}
    \caption{The path in $X$, in the case $m=3$.}
    \label{path}
\end{figure}

We have
\[(s_0\ldots s_{m-1})P_0=(s_0\ldots s_{m-1})Q_m=gP_0\]
and thus $g=s_0\ldots s_{m-1} z_m s_m$ with $s_m\in G_{P_0}$. The $z_i$ and $s_i$ define another reduced word presentation of $g$. It follows that $m=n$, $z_i=y_i$ and $Q_i=P_i$.
The $r_i$ and $s_i$ are linked via the equivalence relation generated by \eqref{equivalence-reduced-words}. But it is easy to check that the statement of the lemma is insensitive to this equivalence relation, and thus we assume $r_i=s_i$ without loss of generality.

Consequently, $F$ fixes the edge $|_i c, _i \mu| y_{i+1}$ and is thus contained in its stabiliser, which yields the claim of the lemma.
\end{proof}

Lemma~\ref{normalising-circle} can be used to check whether the normaliser $N_\pi(F)$ of $F$ is infinite, for example in conjunction with the equivalence of the first two items of  Proposition~\ref{infinite-normaliser} below.

\begin{exam}
\label{example:edge}
Let $Y$ consist of an edge $y$, with vertex groups $A$ and $B$ and edge group $C$. We identify $C$ with both its images in $A$ and $B$. Let $F\subseteq A$.

We can draw the following conclusions from the above lemma:
\begin{itemize}
    \item If $F$ is not subconjugate to $C$ in $A$, then $N_\pi(F)=N_A(F)$.
    \item If $F\subseteq C$ and there exist $a\in N_A(F)\setminus C$ and $b\in N_B(F)\setminus C$, then $N_\pi(F)$ contains the element $1yb\overline{y}a$ of infinite order.
\end{itemize}
We emphasise that the second condition is \emph{not} necessary for the infinity of the normaliser. For example, let $A=B$ equal the dihedral group \[D_8=\langle \sigma, \tau\mid \sigma^4=\tau^2=(\sigma\tau)^2\rangle\] and let \[C=\langle \tau, \sigma^2\tau\rangle = \{ 1, \tau, \sigma^2\tau, \sigma^2\}\,.\] Let $F$ be the $2$-element subgroup generated by $\tau$. Then $N_B(F)=C$, so the second item cannot be satisfied. However, the normaliser of $F$ contains the infinite order element given by the reduced word $1y\sigma\overline{y}\sigma^{-1}$.

What happens here is that $F=\langle \tau \rangle$ is conjugated in the second step into $\langle \sigma^2\tau\rangle$ which still lies in $C$, and then back into $F$ in the third step. Similarly, there can be situations when a chain of length two doesn't suffice, but length $3$ or higher is necessary.
\end{exam}

\begin{exam}
Let $Y$ consist of a loop $t$, with vertex group $A$ and loop group $C$. As usual, we identify $C$ with one of its images in $A$ and denote the other by $\iota(C)$.
We can draw the following conclusions from the above lemma:
\begin{itemize}
    \item If $F$ is neither subconjugate to $C$ nor to $\iota(C)$ in $A$, then $N_\pi(F)=N_A(F)$.
    \item If $F\subseteq C$ and there exists $a\in \mathrm{Trans}_A(\iota(F), F)$, then $N_\pi(F)$ contains the element $1ta$ of infinite order.
\end{itemize}
Again, the second condition is not the only way to produce an infinite order element: Let $A=C$ equal the Klein $4$-group $\{1, a, b, c\}$, and let $\iota\colon C \rightarrow C$ be the order $3$ automorphism mapping $a$ to $b$, $b$ to $c$ and $c$ to $a$. Then $F=\langle a \rangle$ cannot be normalised by going around the loop once -- note that all conjugations in $A$ are trivial --, but the (infinite order) element $t^3=1t1t1t1$ normalises $F$.
\end{exam}

A \emph{ray} in $X$ is a geodesic embedding of the metric space $[0,\infty)$ (with the standard metric) into $X$, and a \emph{line} in $X$ is a geodesic embedding $\mathbb{R}\hookrightarrow X$. The existence of a $\mathrm{CAT}(0)$ metric on $X$ equips it with a Gromov boundary. The underlying set can be described as the set of all rays emanating from a fixed point $p$ \cite[Lemma~III.H.3.1]{bridson-haefliger}. Since $X$ is a tree, there is no need of an equivalence relation on the rays. Topologically, $\partial X$ is a Cantor space.

\begin{prop}\label{infinite-normaliser}
Let $F$ be a finite subgroup of $\pi$. Then the following are equivalent:
\begin{itemize}
    \item[(i)] $N_\pi(F)$ is infinite,
    \item[(ii)] $N_\pi(F)$ contains an element of infinite order,
    \item[(iii)] $X^F$ is an infinite graph,
    \item[(iv)] $F$ fixes a ray in the tree $X$ pointwise,
    \item[(v)] $F$ fixes a line in the tree $X$ pointwise,
    \item[(vi)] $F$ fixes a point on the Gromov boundary $\partial X$,
    \item[(vii)] $F$ fixes two points on the Gromov boundary $\partial X$.
\end{itemize}
\end{prop}

\begin{proof} \emph{(i) $\Rightarrow$ (ii).} $N_\pi(F)$ acts on the tree $X$ and is thus isomorphic to the fundamental group of a certain connected graph of finite groups $(G',Y')$. This need not be a finite graph of groups, but it inherits from $(G,Y)$ the property that there is a global bound on the orders of the groups $G'_P$. We show that this suffices for the existence of an element of infinite order if $N_\pi(F)$ is infinite.

Assume that $N_\pi(F)$ is a torsion group. The fundamental group of the graph of groups $(G',Y')$ surjects onto the usual topological fundamental group of $Y'$, which has an element of infinite order unless $Y'$ is a tree. Moreover, any edge group has to equal the vertex groups of one of  its two vertices. The reason is that there certainly is an element of infinite order if both inclusions are strict, see Example~\ref{example:edge} with $F=\{1\}$. Note that the fundamental group of the graph of groups on any subgraph of $Y'$ embeds into the whole fundamental group, as can be seen by considering reduced words. Now, let $P$ be a vertex of $Y'$ such that $G'_P$ is of maximal order. Then the two facts mentioned above ensure that the canonical map \[G'_P\longrightarrow \pi_1(G', Y', P) \cong N_\pi(F)\]
is an isomorphism, thus $N_\pi(F)$ is finite.

\emph{(ii) $\Rightarrow$ (iii).}
Since $F$ is finite, it fixes a vertex $x$. Let $g\in N_\pi(F)$ be of infinite order. For all $n$,  we have $g^{-n}Fg^n\subseteq \mathrm{Stab}_\pi(x)$, or, equivalently, $F\subseteq \mathrm{Stab}_\pi(g^nx)$. Since  the stabiliser of $x$ is finite, there are infinitely many points of the form $g^n x$.

\emph{(iii) $\Rightarrow$ (i).} Since there are finitely many $\pi$-orbits, $F$ in particular fixes infinitely many vertices in the same $\pi$-orbit. Suppose that this is the $G$-orbit of $x$, i.\ e. there are infinitely many $g\in \pi$ such that
\[F\subseteq \mathrm{Stab}_\pi(gx)\]
or, equivalently,
\[ g^{-1} F g\subseteq \mathrm{Stab}_\pi(x)\,.\]
Thus, the subgroup $M\subseteq \pi$ consisting of all $g$ with the above property is infinite. But $M$ acts on the finite set of $\pi$-conjugates of $F$ in $\mathrm{Stab}_{\pi}(x)$ by conjugation, and the stabiliser of $F$ equals $N_\pi(F)$ which is thus also infinite.

\emph{(iii) $\Leftrightarrow$ (iv).}
In the tree $X$, every vertex is of finite degree since this is true for the quotient $Y=\pi\backslash X$ and all edge stabilisers are finite. The same is thus true for $X^F$ which is connected by uniqueness of geodesics. Finally, a connected graph in which all vertices have finite degree is infinite if and only if it contains a ray \cite[Prop.~8.2.1]{diestel}.

\emph{(iv) $\Leftrightarrow$ (vi), (v) $\Leftrightarrow$ (vii).} In the description of the Gromov boundary as the set of rays emanating from a fixed vertex $p$ recalled directly before this Proposition, take $p$ to be a vertex fixed by $F$. Then $F$ acts on the set of rays, and the assertions translate into one another.

\emph{(ii) $\Rightarrow$ (v).} Let $g$ be of infinite order normalising $F$ and let $x$ be a vertex fixed by $F$. Then $F$ fixes all $g^nx$ with $n\in \mathbb{Z}$. Let
\[m = \mathrm{min}_{x \in \mathrm{vert} X} d(x,gx) >0\,.\]
By the structure theorem for hyperbolic elements \cite[Prop.~24]{trees}, there is a $g$-invariant line $L$ on which $g$ acts by translation by $m$. Moreover, if $z$ denotes the point on $L$ closest to $x$, then the geodesic from $x$ to $gx$ contains $z$ (and $gz$). Thus, $z\in X^F$. A similar argument applied to $g^n x$ shows that $g^nz$ is contained in $X^F$. Since $X^F$ is geodesically closed, this implies that it contains the whole line $L$.

\emph{(v) $\Rightarrow$ (iv).} This is a tautology.
\end{proof}

\end{appendix}

\end{document}